\documentclass[12pt]{amsart}
\usepackage{amsmath}
\usepackage{amsthm}
\usepackage{amssymb}
\usepackage{amscd}
\usepackage{hyperref}
\usepackage{tikz}
\usetikzlibrary{matrix,arrows}
\usepackage{mabliautoref}

\DeclareMathOperator{\lc}{H}

\newcommand{\hght}{\operatorname{ht}}
\newcommand{\Spec}{\operatorname{Spec}}
\newcommand{\length}{\ell}

\newcommand{\red}[1]{#1_{\operatorname {red}}}

\newcommand{\eh}{\operatorname{e}}
\newcommand{\ehk}{\eh_{\text{HK}}}

\newcommand{\mf}{\mathfrak}

\DeclareMathOperator{\m}{\mathfrak{m}}
\DeclareMathOperator{\n}{\mathfrak{n}}
\DeclareMathOperator{\ord}{ord}

\DeclareMathOperator{\im}{im}

\DeclareMathOperator{\numg}{\mu}
\DeclareMathOperator{\rank}{rank}
\DeclareMathOperator{\inm}{in_{\m}}
\DeclareMathOperator{\inte}{in_<}

\newcommand{\gr}{\operatorname{gr}}

\begin{document}

\title{Asymptotic Lech's inequality}

\author{Craig Huneke}
\address{Department of Mathematics, University of Virginia, Charlottesville, VA 22904 USA}
\email{clh4xd@virginia.edu}

\author{Linquan Ma}
\address{Department of Mathematics, Purdue University, West Lafayette, IN 47907 USA}
\email{ma326@purdue.edu}

\author{Pham Hung Quy}
\address{Department of Mathematics, FPT University, Hanoi, Vietnam}
\email{quyph@fe.edu.vn}

\author{Ilya Smirnov}
\address{Department of Mathematics, Stockholm University, S-10691, Stockholm, Sweden}
\email{smirnov@math.su.se}

\begin{abstract}
We explore the classical Lech's inequality relating the Hilbert--Samuel multiplicity and colength of an $\m$-primary ideal in a Noetherian local ring $(R,\m)$. We prove optimal versions of Lech's inequality for sufficiently deep ideals in characteristic $p>0$, and we conjecture that they hold in all characteristics.

Our main technical result shows that if $(R,\m)$ has characteristic $p>0$ and $\widehat{R}$ is reduced, equidimensional, and has an isolated singularity, then for any sufficiently deep $\m$-primary ideal $I$, the colength and Hilbert--Kunz multiplicity of $I$ are sufficiently close to each other. More precisely, for all $\varepsilon>0$, there exists $N\gg0$ such that for any $I\subseteq R$ with $\length(R/I)>N$, we have $(1-\varepsilon)\length(R/I)\leq \ehk(I)\leq(1+\varepsilon)\length(R/I)$.
\end{abstract}

 \dedicatory{Dedicated to Professor Bernd Ulrich on the occasion of his 65th birthday}

\maketitle

\setcounter{tocdepth}{1}
\tableofcontents

\section{Introduction and preliminaries}
In \cite{LechMultiplicity}, Lech proved a simple inequality relating the Hilbert--Samuel multiplicity (Definition~\ref{def HS}) and the colength of an ideal:
\begin{theorem}[Lech's inequality]
\label{thm.Lech}
Let $(R,\m)$ be a Noetherian local ring of dimension $d$ and let $I$ be any $\m$-primary ideal of $R$. Then we have $$\eh(I)\leq d!\eh(R)\length(R/I),$$ where $\eh(I)$ denotes the Hilbert--Samuel multiplicity of $I$ and $\eh(R)=\eh(\m)$.
\end{theorem}

Suppose $R$ is a regular local ring and $I=J^t$, a power of an $\m$-primary ideal $J$, then it is easy to see that when $t\to\infty$, both sides of Lech's inequality tend to $t^d\eh (J)$. In particular, \autoref{thm.Lech} is asymptotically sharp when $R$ is regular. However, Lech also observed in the proof \cite[page 74, after (4.1)]{LechMultiplicity} that the inequality in \autoref{thm.Lech} is almost never sharp: when $d > 1$, we always have a strict inequality.

In \cite{MumfordStability}, Mumford conceptualized this and defined the $0$-th flat multiplicity of a Noetherian local ring $R$ of dimension $d$ to be
\[
\eh_0(R) =  \sup_{\substack{\sqrt{I}=\m \\}} \left\{\frac{\eh(I)}{d!\length(R/I)} \right\},
\]
and the $k$-th flat multiplicity of $R$ to be $\eh_k(R)=\eh_0(R[[t_1,\dots,t_k]])$. He proved that $\eh_0(R)\geq \eh_1(R)\geq \cdots \geq \eh_k(R)\geq \cdots$, and defined a Noetherian local ring to be {\it semi-stable} if $\eh_1(R) = 1$. However, computing $\eh_k(R)$ (or even determining whether $R$ is semi-stable) turns out to be difficult (see \cite[Section 3]{MumfordStability}). In this paper, we study an {\it asymptotic} version of Mumford's $\eh_0(R)$. Our main purpose is to obtain sharp versions of Lech's inequality for {\it sufficiently deep} ideals beyond the regular case. Our main conjecture is the following:

\begin{conjecture}[Asymptotic Lech's inequality]
\label{conj.asymptoticLech}
Let $(R,\m)$ be a Noetherian local ring of dimension $d\geq 1$.
\begin{enumerate}
\item If $\widehat{R}$ has an isolated singularity, i.e.,  $\widehat{R}_P$ is regular for all $P\in\Spec\widehat{R}-\{\m\}$, then
\[
\lim_{N\to\infty} \sup_{\substack{\sqrt{I}=\m \\ \length(R/I)> N}} \left\{\frac{\eh(I)}{d!\length(R/I)} \right\}=1.
\]
As the above limit is always $\geq 1$ by considering $I=\m^n$ and letting $n\to\infty$, the statement is equivalent to saying that for every $\varepsilon>0$, there exists $N\gg0$ such that for any $\m$-primary ideal $I$ with $\length(R/I)>N$,
$$\eh(I)\leq d!(1+\varepsilon)\length(R/I).$$
\item We have $\eh(\red{\widehat{R}}) > 1$ if and only if
\[
\lim_{N\to\infty} \sup_{\substack{\sqrt{I}=\m \\ \length(R/I)> N}} \left\{\frac{\eh(I)}{d!\length(R/I)} \right\}<\eh(R).
\]
In other words,  $\eh(\red{\widehat{R}}) > 1$ if and only if there exists $\varepsilon>0$ and $N\gg0$ such that for any $\m$-primary ideal $I$ with $\length(R/I)>N$,
$$\eh(I)\leq d!(\eh(R)-\varepsilon)\length(R/I).$$
\end{enumerate}
\end{conjecture}

Compared with Lech's inequality, \autoref{conj.asymptoticLech} expects that for ideals having large colength
the constant $\eh(R)$ in \autoref{thm.Lech} can be usually replaced by a much smaller number under various assumptions on the ring $R$. Our main result is the following:

\begin{theoremA*}[\autoref{cor.mainIsolatedSing}, \autoref{prop.if direction}, \autoref{cor.Mainequalchar}, \autoref{prop.MaindimOne}] \hspace{2em}
\begin{enumerate}
\item[(1)] \autoref{conj.asymptoticLech} (a) holds in characteristic $p>0$ when $R/\m$ is perfect.
\item[(2)] \autoref{conj.asymptoticLech} (b) holds in equal characteristic.
\end{enumerate}
\end{theoremA*}

We also show that in \autoref{conj.asymptoticLech} part (a), the assumption that $\widehat{R}$ has an isolated singularity is necessary in general: we construct a counter-example for non-isolated singularities in \autoref{example}. As for \autoref{conj.asymptoticLech} part (b), note that in characteristic $p>0$, we have
\begin{equation}
\label{eqn:dag}
\frac{\eh(I)}{d!\length(R/I)}\leq \frac{\ehk(I)}{\length(R/I)}\leq \ehk(R) \tag{\dag}
\end{equation}
by \cite[Lemma~4.2]{WatanabeYoshida}. So  \autoref{conj.asymptoticLech} (b) holds whenever $\ehk (R)<\eh(R)$ (this happens often, for example when $R$ is $F$-rational but not regular). But our Theorem A above proves it even in the case $\ehk (R)= \eh(R)$. Nonetheless, we propose a stronger conjecture in characteristic $p>0$:

\begin{conjecture}
\label{conj.asymptoticLecheHK}
Let $(R,\m)$ be a Noetherian local ring of characteristic $p>0$ and dimension $d \geq 1$. If
$\eh(\red{\widehat{R}}) > 1$, then
\[
\lim_{N\to\infty} \sup_{\substack{\sqrt{I}=\m \\ \length(R/I)> N}} \left\{\frac{\eh(I)}{d!\length(R/I)} \right\}<\ehk(R).
\]
\end{conjecture}

Note that by \autoref{eqn:dag}, the limit in \autoref{conj.asymptoticLecheHK} is always $\leq \ehk(R)$. Our Theorem A implies that \autoref{conj.asymptoticLecheHK} is true at least for isolated singularities with perfect residue fields (see \autoref{cor.maineHK}). Besides being natural, our conjectures are largely inspired from recent work of Blum and Liu \cite[Lemma 13]{BlumLiu}:
\begin{lemma}[Blum--Liu]
\label{lem.Blum-Liu}
Let $(R,\m)$ be a Noetherian local ring of dimension $d$ such that $\widehat{R}$ is a domain and $R/\m$ is algebraically closed. For any positive numbers $\delta,\varepsilon \in (0,1)$, there exists $n_0$ such that for any $n\geq n_0$ and any ideal $\m^n\subseteq I\subseteq \m^{\lceil\delta n\rceil}$ we have $$\eh(I)\leq d!(1+\varepsilon)\length(R/I).$$
\end{lemma}
This Lech-type inequality, although a bit technical, is a crucial ingredient in their proof of semicontinuity of normalized volume function on the valuation space centered at the origin, see \cite{BlumLiu}. The main difference between \autoref{lem.Blum-Liu} and \autoref{conj.asymptoticLech} part (a) is that in the latter we no longer require the additional parameter $\delta$ and we expect the inequality for any $I$ that has sufficiently large colength. This is clearly much stronger than the conclusion of \autoref{lem.Blum-Liu}, but assumes a strong hypothesis of isolated singularity. The proof of \autoref{lem.Blum-Liu} given in \cite{BlumLiu} is rather involved and makes use of local Okounkov bodies. We will give two elementary approaches in characteristic $p>0$ using Hilbert--Kunz multiplicities in section 2, which in fact generalizes their result (in characteristic $p>0$). More importantly, one of the approaches will eventually lead to the main technical theorem of this article, from which part of Theorem A will follow.

\begin{theoremB*}[\autoref{thm.mainHK}]
Let $(R,\m)$ be a Noetherian local ring of characteristic $p>0$ and dimension $d\geq 1$ with $K=R/\m$ perfect. Suppose $\widehat{R}$ is reduced, equidimensional, and has an isolated singularity. Then for every $\varepsilon>0$, there exists $N\gg0$ such that for any $\m$-primary ideal $I$ with $\length(R/I)>N$,
$$(1-\varepsilon)\length(R/I)\leq \ehk (I)\leq (1+\varepsilon)\length(R/I).$$
\end{theoremB*}

This will be proven in section 4. Note that if $R$ is regular, then $\ehk(I)=\length(R/I)$ for any $\m$-primary ideal $I$ so the theorem is trivial in this case. Basically, Theorem B indicates that, for isolated singularities, the relation between colength and Hilbert--Kunz multiplicities for sufficiently deep $\m$-primary ideals is somehow similar to the case of regular rings.

\subsection{Hilbert--Samuel and Hilbert--Kunz multiplicity}

\begin{definition}\label{def HS}
Let $(R, \mf m)$ be a Noetherian local ring of dimension $d$ and $I$ be an $\mf m$-primary ideal.
The Hilbert--Samuel multiplicity of $I$ is defined as
\[
\eh(I) = \lim_{n\to\infty}\frac{d!\length(R/I^n)}{n^d}.
\]
\end{definition}

A closely related concept is integral closure. An element $x\in R$ is integral over an ideal $I$ if
it satisfies an equation of the form $x^n + a_{1}x^{n-1} + \cdots a_{n-1}x+ a_n=0$
where $a_k \in I^k$. The set of all elements $x$ integral over $I$ is an ideal and is denoted $\overline{I}$.
The Hilbert--Samuel multiplicity is an invariant of the integral closure, i.e., $\eh(I) = \eh(\overline{I})$.
Even more generally, if $\widehat{R}$ is reduced, then
\[
\eh(I) =\lim_{n\to\infty}\frac{d!\length(R/\overline{I^n})}{n^d},
\]
because there exists $k$ such that $\overline{I^{nk}} = (\overline{I^k})^n$ for any $n \geq 1$ by
\cite[Corollary~4.13]{Ratliff2}.

\begin{definition}
Let $(R, \mf m)$ be a Noetherian local ring of characteristic $p> 0$ and dimension $d$, and let $I$ be an $\mf m$-primary ideal.
The Hilbert--Kunz multiplicity of $I$ is defined as
\[
\ehk(I) = \lim_{e\to\infty}\frac{\length(R/I^{[p^e]})}{p^{ed}},
\]
where $I^{[p^e]}$ is the ideal generated by $p^e$-th powers of elements of $I$.
\end{definition}

It is a nontrivial result of Monsky \cite{Monsky} that the above limit exists. We point out that it follows from work of Kunz \cite{Kunz2} that, if $R$ is reduced and $F$-finite (i.e., the Frobenius map $R\xrightarrow{F}R$ is a finite map), then $$\ehk(I)=\lim_{e\to\infty}\frac{\length(R^{1/p^e}/IR^{1/p^e})}{p^{e\gamma}}$$ where $\gamma=d+\log_p[K^{1/p}:K]$ for $K=R/\m$. In general, the two multiplicities are related by the inequalities (\cite[(2.4)]{WatanabeYoshida}):
$$\frac{\eh(I)}{d!}\leq \ehk(I)\leq \eh(I).$$
It follows that $\eh(I)=\ehk(I)$ when $d\leq 1$. On the other hand, as long as $d\geq 2$, the Hilbert--Kunz multiplicity of $I$ can be close to either ${\eh(I)}/{d!}$ or $\eh(I)$. For example, if $I=(x_1,\dots,x_d)$ is generated by a system of parameters of $R$, then we always have $\ehk(I)=\eh(I)$ (\cite[Theorem~2]{Lech2}), while if $I=J^n$ is a power of an $\m$-primary ideal $J$, we have \cite[Theorem 1.1]{WatanabeYoshidaHKdimTwo}:
\begin{equation}
\label{eqn:asympHKHS}
\lim_{n\to\infty}\frac{\ehk(J^n)}{\eh(J^n)/d!}=\lim_{n\to\infty}\frac{\ehk(J^n)}{\length(R/J^n)}=1.
\end{equation}
We will strengthen this result by showing that, if $\widehat{R}$ is reduced and equidimensional, then $\lim\limits_{n\to\infty}\frac{\ehk(J^n)}{\length(R/J^n)}$ converges to $1$ uniformly, independent of $J$ (see \autoref{rem.UniformeHKI^n}).

\vspace{1em}

\noindent\textbf{Acknowledgement}: The second author thanks Kazuma Shimomoto and Bernd Ulrich for comments on the Cohen--Gabber theorem, and he thanks Yuchen Liu for valuable discussions. The second author was supported in part by NSF Grant DMS \#1901672, NSF FRG Grant DMS \#1952366, and was supported in part by NSF Grant DMS $\#1836867/1600198$ when preparing this article. The third author is partially supported by a fund of Vietnam National Foundation for Science and Technology Development (NAFOSTED) under grant number 101.04-2017.10. The authors would also like to thank the referee for her/his comments that leads to improvement of this article.

\section{Blum--Liu's lemma in characteristic $p>0$}
In this section we give two simple alternative proofs of \autoref{lem.Blum-Liu} in characteristic $p>0$. These results are not used directly in the proof of the main theorems. However, the methods inspired the strategy of the proof of the main result in section 4, and we believe they have independent interest.

We begin by recalling a lemma which is due to Watanabe \cite[Theorem 2.1]{WatanabeChain} when $R$ is complete normal with algebraically closed residue field. But the conclusion holds for any complete local domain with algebraically closed residue field, which is implicit in the proof of \cite[Lemma 3.1]{KMQSY}. For the sake of completeness we give the argument.

\begin{lemma}[Watanabe]
\label{lem.Watanabe}
Let $(R,\m)$ be a Noetherian complete local domain with $K=R/\m$ algebraically closed. Then for $I\subsetneq J$ two integrally closed $\m$-primary ideals, we can find a chain $$I=I_0\subseteq I_1\subseteq\cdots\subseteq I_n=J$$ such that $\length(I_j/I_{j-1})=1$ for every $j$ and all $I_j$ are integrally closed.
\end{lemma}
\begin{proof}
By induction on $\length(J/I)$, it is enough to find an integrally closed ideal $I'$ such that $I\subseteq I'\subseteq J$ and $\length(I'/I)=1$. Let $R\to S$ be the normalization of $R$. Since $R$ is a complete local domain, $S$ is local by \cite[Proposition 4.8.2]{HunekeSwanson}, and so $S=(S,\n)$ is a normal local domain with $R/\m=S/\n=K$ since $K$ is algebraically closed. In particular, computing length over $R$ and $S$ are the same. By \cite[Theorem 2.1]{WatanabeChain}, there exists a chain $$\overline{IS}=J_0\subseteq J_1\subseteq J_2\subseteq\cdots \subseteq J_m=\overline{JS}$$ such that each $J_i$ is integrally closed in $S$ and $\length(J_{i+1}/J_i)=1$ for every $i$.

Since $I$ is integrally closed in $R$ and $S$ is integral over $R$, by \cite[Proposition 1.6.1]{HunekeSwanson} we know
$$J_0\cap R=\overline{{IS}}\cap R=\bar{I}=I,$$ and similarly we know that $J_m\cap R=J$.
Let $t=\max\{i \mid J_i\cap R=I\}$.  Obviously $0\leq t<m$. Set $I'=J_{t+1}\cap R$.  Now we have $I\subseteq I'\subseteq J$ and $I'$ is integrally closed in $R$ (one can use \cite[Proposition 1.6.1]{HunekeSwanson} again). Moreover, $\length(I'/I)>0$ by our choice of $t$ while $I'/I\hookrightarrow J_{t+1}/J_t$ shows that $\length(I'/I)\leq \length(J_{t+1}/J_t)=1$. Thus, we have $\length(I'/I)=1$.
\end{proof}

\begin{proof}[Proof of \autoref{lem.Blum-Liu} in characteristic $p>0$]
We can pass to the completion of $R$ to assume $R$ is a complete local domain with $R/\m$ algebraically closed. We fix $\varepsilon'\leq(\delta^d\varepsilon)/2$. We next pick $n_0$ such that for any $n\geq n_0$,
\begin{enumerate}
  \item $\ehk(\m^n)\leq (1+\varepsilon')\length(R/\overline{\m^n})$.
  \item $\length(R/\overline{\m^n})/{\length(R/\overline{\m^{\lceil\delta n\rceil}})}\leq \frac{2}{\delta^d}$.
\end{enumerate}
We note that such $n_0$ exists: we have $\lim\limits_{n\to\infty}\frac{\ehk(\m^n)}{\length(R/\m^n)}=1$ by \autoref{eqn:asympHKHS} and $\lim\limits_{n\to\infty}\frac{\length(R/\m^n)}{\length(R/\overline{\m^n})}=1$, which guarantees (a), and we can achieve (b) by using that
\[\lim_{n\to\infty}\length(R/\overline{\m^n})/{\length(R/\overline{\m^{\lceil\delta n\rceil}})}=\lim_{n\to\infty}\frac{d!\eh(R)n^d}{d!\eh(R)(\lceil\delta n\rceil)^d}=\frac{1}{\delta^d}.
\]
Since $R$ is a complete local domain with algebraically closed residue field, $R^{1/p^e}$ is a finite $R$-module of rank $p^{ed}$. Given an ideal $I$ such that $\m^n\subseteq I\subseteq \m^{\lceil\delta n\rceil}$, we apply \cite[Corollary 2.2]{CDHZ} and \autoref{lem.Watanabe} to verify that
$$\length(R^{1/p^e}/\overline{\m^n}R^{1/p^e})-\length(R^{1/p^e}/\overline{I}R^{1/p^e})=\length\left(\frac{\overline{I}R^{1/p^e}}{\overline{\m^n}R^{1/p^e}}\right)\geq \length(\overline{I}/\overline{\m^n})\cdot p^{ed}.$$
After dividing by $p^{ed}$ and letting $e\to \infty$, we obtain
$$\ehk(\overline{I})\leq \ehk(\overline{\m^n})-\length\left (\overline{I}/\overline{\m^n}\right)\leq (1+\varepsilon')\length(R/\overline{\m^n})-\length\left (\overline{I}/\overline{\m^n}\right)=\varepsilon'\length(R/\overline{\m^n})+\length(R/\overline{I}),$$ where we used (a) for the inequality in the middle. Finally, we divide the above equation by $\length(R/\overline{I})$ to bound
\begin{eqnarray*}
  \frac{\eh(I)}{d!\length(R/I)} &\leq& \frac{\eh(\overline{I})}{d!\length(R/\overline{I})}\leq \frac{\ehk(\overline{I})}{\length(R/\overline{I})} \\
   &\leq & 1+\varepsilon'\frac{\length(R/\overline{\m^n})}{\length(R/\overline{I})} \leq 1+\varepsilon'\frac{\length(R/\overline{\m^n})}{\length(R/\overline{\m^{\lceil\delta n\rceil}})}\\
   &\leq& 1+\varepsilon'\cdot\frac{2}{\delta^d}\leq 1+\varepsilon,
\end{eqnarray*}
where we used (b) and the choice of $\varepsilon'$. This completes the proof.
\end{proof}

Our next proposition is a generalization of \autoref{lem.Blum-Liu} in characteristic $p>0$: besides establishing an upper bound on the Hilbert--Kunz multiplicity in terms of colength, we also obtain a lower bound, and we can relax the assumptions on $R$. We will later see what are the optimal assumptions in \autoref{rem.OptimalBlumLiu}. More importantly, the proof strategy will be adapted and extended to prove Theorem B (after we established some uniformity results in section 3).

\begin{proposition}
\label{prop.BL-2ndproof}
Let $(R,\m, k)$ be a Noetherian local ring of dimension $d$ of characteristic $p>0$ such that $\widehat{R}$ is reduced and equidimensional. Then for any positive numbers $\delta,\varepsilon \in (0,1)$, there exists $n_0$ such that for any $n\geq n_0$ and any ideal $\m^n\subseteq I\subseteq \m^{\lceil\delta n\rceil}$, we have $$(1-\varepsilon)\length(R/I)\leq \ehk(I)\leq (1+\varepsilon)\length(R/I).$$ As a consequence, $(1-\varepsilon)\length(R/I)\leq \eh(I)\leq d!(1+\varepsilon)\length(R/I)$.
\end{proposition}
\begin{proof}
It is easy to see that, if we can prove the proposition for a faithfully flat extension $R'$ of $R$ with $R'/\m R'$ a field, then the same conclusion holds for $R$ since both colength and multiplicities do not change when we pass to $R'$. With this in mind, we can first complete $R$ to assume $R$ is a complete local ring that is reduced and equidimensional. We next apply Hochster--Huneke's $\Gamma$-construction to reduce to the case $R$ is $F$-finite. For any cofinite subset $\Gamma$ of a $p$-base of the residue field $k$, we have a faithfully flat and purely inseparable extension $R\to R^\Gamma$ (see \cite[(6.11)]{HochsterHuneke2} for details), and when $\Gamma$ is chosen to be sufficiently small, $R^\Gamma$ is still reduced by \cite[Lemma 6.13]{HochsterHuneke2} and equidimensional (since it is purely inseparable and $R$ is equidimensional). Finally, we replace $R^\Gamma$ by $\widehat{R^\Gamma}$ to assume $R$ is complete, $F$-finite, reduced and equidimensional (note that $R^\Gamma$ is excellent so completion preserves these properties). Under these conditions, the total quotient ring of $R$ is a product of fields $\prod_{i=1}^nK_i$, where $K_i=\text{Frac}(R/P_i)$ for each minimal prime $P_i$ of $R$.
Since $R$ is equidimensional, by \cite[Proposition~2.3]{Kunz2}
for  $\log_p[K_i^{1/p}:K_i]=d+\log_p[k^{1/p}:k]$ for all $i$. We call this constant $\gamma$.

By the Cohen--Gabber theorem (see \cite[Expos\'{e} IV, Th\'{e}or\`{e}me 2.1.1]{CohenGabberOriginal} or \cite[Theorem 1.1]{KuranoShimomoto}), there exists a complete regular local ring $A$ with the same residue field as $R$ such that $A\to R$ is module-finite and generically \'{e}tale. Let $c\neq 0$ be the discriminant of the map $A\to R$, then by \cite[Discussion 6.3 and Lemma 6.5]{HochsterHuneke1}, $cR^{1/p^e}\subseteq R[A^{1/p^e}]\cong R\otimes A^{1/p^e}$ for any $e$. We consider the following two short exact sequences:
$$0\to R[A^{1/p^e}]\to R^{1/p^e}\to C_e\to 0$$
$$0\to R^{1/p^e}\xrightarrow{\cdot c}R[A^{1/p^e}]\to D_e\to 0.$$
It follows that $C_e$ and $D_e$ are annihilated by $c$ for any $e$, and $\mu(C_e)\leq \mu(R^{1/p^e})$, $\mu(D_e)\leq \mu(R[A^{1/p^e}])$. Thus we have surjections:
\begin{equation}
\label{eqn:TwoSurj}
(R/cR)^{\mu(R^{1/p^e})}\twoheadrightarrow C_e\to 0, \text{ and } (R/cR)^{\mu(R[A^{1/p^e}])}\twoheadrightarrow D_e\to 0.
\end{equation}
For any $\m$-primary ideal $I$ we tensor the two short exact sequences with $R/I$. Since $A$ is regular local, $R[A^{1/p^e}]\cong R\otimes A^{1/p^e}\cong R^{p^{e\gamma}}$ by \cite{Kunz1}. Thus we have
$$(R/I)^{p^{e\gamma}}\to R^{1/p^e}/IR^{1/p^e}\to C_e/IC_e\to 0$$
$$R^{1/p^e}/IR^{1/p^e}\to(R/I)^{p^{e\gamma}}\to D_e/ID_e\to 0.$$
Computing length, we know that
$$p^{e\gamma}\length(R/I)-\length(D_e/ID_e)\leq \length(R^{1/p^e}/IR^{1/p^e})\leq p^{e\gamma}\length(R/I)+\length(C_e/IC_e),$$
so by (\ref{eqn:TwoSurj})
$$p^{e\gamma}\length(R/I)-p^{e\gamma}\length(R/(I,c))\leq \length(R^{1/p^e}/IR^{1/p^e}) \leq p^{e\gamma}\length(R/I)+\mu(R^{1/p^e})\length(R/(I,c)).$$
Dividing the above by $p^{e\gamma}$, we get that for any $\m$-primary ideal $I$ and every $e$,
\begin{equation}
\label{eqn:InequaAllIe}
\length(R/I)\left(1-\frac{\length(R/(I,c))}{\length(R/I)}\right)\leq \frac{\length(R^{1/p^e}/IR^{1/p^e})}{p^{e\gamma}}\leq \length(R/I)\left(1+\frac{\mu(R^{1/p^e})}{p^{e\gamma}}\cdot\frac{\length(R/(I,c))}{\length(R/I)}\right).
\end{equation}

At this point, we note that if $\m^n\subseteq I\subseteq \m^{\lceil\delta n\rceil}$, then we have
$$\frac{\length(R/(I,c))}{\length(R/I)}\leq \frac{\length(R/(\m^n,c))}{\length(R/\m^{\lceil\delta n\rceil})}.$$
But when $n\to \infty$, we know that
$$\lim_{n\to \infty} \frac{\length(R/(\m^n,c))}{\length(R/\m^{\lceil\delta n\rceil})}=\lim_{n\to \infty}\frac{\eh(R/cR)n^{d-1}/(d-1)!}{\eh(R)\delta^dn^d/d!}=\lim_{n\to \infty}\frac{d\eh(R/cR)}{\eh(R)\delta^d}\cdot \frac{1}{n}=0,$$ since $\frac{d\eh(R/cR)}{\eh(R)\delta^d}$ is a constant that does not depend on $n$. Therefore we know that for every $0<\varepsilon<1$ and every $0<\delta<1$, there exists $n_0$ such that for any $n>n_0$ and any $\m^n\subseteq I\subseteq \m^{\lceil\delta n\rceil}$, we have
$$\frac{\length(R/(I,c))}{\length(R/I)}\leq \frac{\varepsilon}{\ehk(R)}\leq \varepsilon.$$
We plug in the above into (\ref{eqn:InequaAllIe}) and see that for any $\m^n\subseteq I\subseteq \m^{\lceil\delta n\rceil}$ and all $e$, we have
$$\length(R/I)(1-\varepsilon)\leq \frac{\length(R^{1/p^e}/IR^{1/p^e})}{p^{e\gamma}}\leq \length(R/I)\left(1+\frac{\mu(R^{1/p^e})}{p^{e\gamma}}\cdot\frac{\varepsilon}{\ehk(R)}\right).$$
Finally, we take the limit as $e\to\infty$ and use the definition of Hilbert--Kunz multiplicity, we see that
$$(1-\varepsilon)\length(R/I)\leq \ehk(I)\leq (1+\varepsilon)\length(R/I)$$
for any $\m^n\subseteq I\subseteq \m^{\lceil\delta n\rceil}$ as desired. The last conclusion on Hilbert--Samuel multiplicity follows immediately because we have $\ehk(I)\leq \eh(I)\leq d!\ehk(I)$.
\end{proof}

\begin{remark}
\label{rem.UniformeHKI^n}
Let $(R,\m)$ be as in \autoref{prop.BL-2ndproof}. If we apply \autoref{eqn:InequaAllIe} to $I=J^n$ and let $e\to\infty$, we see that $$1-\frac{\length(R/(J^n, c))}{\length(R/J^n)}\leq \frac{\ehk(J^n)}{\length(R/J^n)}\leq 1+\ehk(R)\cdot \frac{\length(R/(J^n, c))}{\length(R/J^n)}.$$
By \autoref{thm.Lech} and \cite[Theorem 2.4]{KMQSY}, we know that there exists a constant $D$ depending only on $R$ and $c$ such that $$\frac{\length(R/(J^n, c))}{\length(R/J^n)}\leq D\cdot \frac{\eh(J^n, R/cR)}{\eh(J^n, R)}=\frac{D}{n}\cdot \frac{\eh(J, R/cR)}{\eh(J, R)}\leq \frac{Dk}{n}$$
where $k$ is the uniform Artin--Rees number for $(c)\subseteq R$, see \cite[Lemma 2.5]{KMQSY}. Therefore as $n\to \infty$, $\frac{\length(R/(J^n, c))}{\length(R/J^n)}\to 0$ uniformly independent of $J$. This shows that, as $n\to\infty$, $\frac{\ehk(J^n)}{\length(R/J^n)}\to 1$ uniformly (independent of the ideal $J$).
\end{remark}

\begin{remark}
If we examine the proof of \autoref{prop.BL-2ndproof} more carefully, the assumption $\m^n\subseteq I\subseteq \m^{\lceil\delta n\rceil}$ is only used to show that $\frac{\length(R/(I,c))}{\length(R/I)}$ is sufficiently close to $0$. At first glance, one might hope that for a fixed $c\neq 0$, $\frac{\length(R/(I,c))}{\length(R/I)}$ always tends to $0$ as long as $\length(R/I)$ tends to infinity (or at least when $I$ is contained in a sufficiently large power of the maximal ideal). If this is indeed the case, then \autoref{conj.asymptoticLech} (a) holds even without the isolated singularity hypothesis on $\widehat{R}$. Unfortunately, this is false in general, as the next example shows.
\end{remark}

\begin{example}
Let $R=K[[x,y]]$ and let $c=x$. Let $I_N = (x^N, x^{N-1}y,\cdots, xy^{N-1}, y^{N^3})$. Then $I_N\subseteq \m^N$, and we have $\length(R/I_N) = N^3 + O(N^2)$ and $\length(R/(I_N, c)) = N^3$. In particular, $$\lim_{N\to\infty}\frac{\length(R/(I_N,c))}{\length(R/I_N)}=1.$$
\end{example}

If we want to prove \autoref{conj.asymptoticLech} (a) using similar strategy, then the subtlety here is that we must choose $c$ such that $\frac{\length(R/(I,c))}{\length(R/I)}$ tends to $0$ for {\it all} sufficiently deep ideals. For this we need $c$ to be sufficiently ``general" (see section 3 and \autoref{clm.Lengthmodc}). But then to run the proof of \autoref{prop.BL-2ndproof} we also require $c$ to be the discriminant of certain map $A\to R$ coming from the Cohen--Gabber theorem. Such $c$ is indeed ``special", and even with the isolated singularity hypothesis, we do {\it not} know whether there exists $c$ that satisfies both conditions. In section 4, we resolve this issue by adjoining invertible indeterminates to $R$ and then applying the Lipman--Sathaye Jacobian theorem \cite{LipmanSathaye,HochsterLipmanSathaye}. This will give us some freedom in choosing $c$.

\begin{remark}
\label{rem.OptimalBlumLiu}
\begin{enumerate}
  \item In \autoref{prop.BL-2ndproof}, for the upper bound $\ehk(I)\leq (1+\varepsilon)\length(R/I)$ (and $\eh(I)\leq d!(1+\varepsilon)\length(R/I)$), we only need to assume $\widehat{R}$ is reduced. In this case we have $0=P_1\cap\cdots P_n\cap Q_1\cap\cdots \cap Q_m$ in $\widehat{R}$ where $P_1,\dots, P_n$ are minimal primes of dimension $d$ and $Q_1,\dots, Q_m$ are minimal primes of lower dimension. Let $S=\widehat{R}/(P_1\cap\cdots \cap P_n)$. Then $S$ is reduced and equidimensional so we can apply \autoref{prop.BL-2ndproof} to $S$. But $\ehk(I)=\ehk(IS)$ since Hilbert--Kunz multiplicity does not see lower-dimensional components, and $\length(S/IS)\leq \length(\widehat{R}/I\widehat{R})=\length(R/I)$. Therefore the result for $S$ implies the result for $R$.
  \item On the other hand, the upper bound in \autoref{prop.BL-2ndproof} fails in general if $\widehat{R}$ is not reduced: let $R=K[[x,y]]/x^2$ and for every $\delta,n$ let $I_n=\m^n+x\m^{\lceil\delta n\rceil}$. Then clearly $\m^n\subseteq I_n\subseteq \m^{\lceil\delta n\rceil}$. Since $x$ is nilpotent, $\overline{I_n}=\overline{\m^n}$ and thus $\eh(I_n)=\eh(\m^n)=2n$ while $\length(R/I_n)=n+\lceil\delta n\rceil$. Therefore for $\varepsilon, \delta$ small, $\eh(I_n)>(1+\varepsilon)\length(R/I_n)$ for all $n$.
\end{enumerate}
\end{remark}

We end this section by showing that, the lower bound for Hilbert--Samuel multiplicity in \autoref{prop.BL-2ndproof} holds in all characteristics in full level of generality.
\begin{proposition}
Let $(R,\m)$ be a Noetherian local ring of dimension $d$. Then for any positive numbers $\delta,\varepsilon \in (0,1)$, there exists $n_0$ such that for any $n\geq n_0$ and any ideal $\m^n\subseteq I\subseteq \m^{\lceil\delta n\rceil}$, we have $$(1-\varepsilon)\length(R/I)\leq \eh(I).$$
\end{proposition}
\begin{proof}
We make use of Vasconcelos's homological degree \cite{Vasconcelos} as in \cite{KMQSY}. We may assume $R$ is complete with infinite residue field. Let $\mathrm{hdeg}(I, R)$ be the homological degree with respect to $I$ (see \cite[Definition 2.3]{KMQSY}). We have $\length(R/I) \le \mathrm{hdeg}(I, R)$ and
$$\mathrm{hdeg}(I, R) = \eh(I) + \sum_{P \in \Lambda} \eh(I, R/P),$$
where $\Lambda$ is a finite set of prime ideals of dimension strictly less than $d$, by \cite[(2.1) and Definition 2.3]{KMQSY}. Since $\m^n\subseteq I\subseteq \m^{\lceil\delta n\rceil}$, we have
$$\eh(I) \geq \eh( \m^{\lceil\delta n\rceil}) = \lceil\delta n\rceil^{d}\eh(R)$$
and
$$\eh(I, R/P) \leq \eh(\m^n, R/P) \le  n^{d-1} \eh(R/P).$$
Therefore we can choose $n_0$ large enough such that $(1-\varepsilon)\mathrm{hdeg}(I, R) \leq \eh(I)$ for any $n\geq n_0$. It follows that $(1-\varepsilon)\length(R/I)\leq \eh(I)$.
\end{proof}

\section{Comparison between socle and colength}

The goal of this section is to prove \autoref{thm.ControlSocle} and \autoref{cor.ControlColon}, which will be used in Section 4. These results are basically saying that the socle of an $\m$-primary ideal $I$ is small compared to the colength of $I$, {\it as long as the colength of $I$ is sufficiently large}. Such a statement is clear if $I$ is a large power of an $\m$-primary ideal, as the growths of the socle and colength are controlled by the Hilbert polynomials. However, we emphasize that we do not impose any condition on the shape of $I$, which is exactly the subtlety.

The proof strategy is to reduce the question to a question about minimal number of generators (instead of socles), and then prove the corresponding statement for all finitely generated modules over a complete regular local ring by induction. Below we give the details. We start by proving a critical lemma for monomial ideals in a polynomial ring.

\begin{lemma}
\label{lem.monomial}
Let $A = K[x_1, \ldots, x_d]$ be a polynomial ring over a field $K$ and $\mf m = (x_1, \ldots, x_d)$.
Then for every $\varepsilon > 0$ there exists $N > 0$ such that for any $\mf m$-primary monomial ideal
$J$ with $\ord (J) \geq N$, we have
\[
\frac {\numg (J)}{\length (A/J)} < \varepsilon.
\]
\end{lemma}
\begin{proof}
We use induction on $d$. The base case $d = 1$ is clear.
Given $\varepsilon >0$, let $k$ be an integer such that $1/k < \varepsilon/2$.
Furthermore, we fix $N_0$ that satisfies the induction hypothesis for $\varepsilon/(2k)$ in $K[x_1, \ldots, x_{d-1}]$.
We claim that $N = N_0 + k$ will satisfy the lemma.

Given an $\mf m$-primary monomial ideal $J \subseteq \mf m^N$, we let $J_i$ be the projection of $J$ onto
$K[x_1, \ldots, x_{d - 1}]$ at $x_d^i$, i.e.,
\[
J_i =
\{x_1^{\alpha_1}\cdots x_{d-1}^{\alpha_{d -1}} \mid x_1^{\alpha_1}\cdots x_{d-1}^{\alpha_{d -1}} x_d^i \in J\}.
\]
Then $\{J_i\}_{i=1}^\infty$ form an ascending chain. We set $A_i = \length (K[x_1, \ldots, x_{d-1}]/J_i)$.

It is easy to see that a minimal generator of $J$ will be necessarily a minimal generator of one of the $J_i$
and no two minimal generators can project to the same monomial (since otherwise they differ by a power of $x_d$ which contradicts that they are both minimal generators). Hence,
\[
\numg(J) \leq \sum_{i = 0}^\infty \numg(J_i) \leq \numg(J_0) + \cdots + \numg (J_{k- 1}) + A_{k-1}.
\]
Since $\length (A/J) \geq A_0 + \cdots + A_{k-1}$, we obtain that
\[
\frac {\numg (J)}{\length (A/J)}
\leq \frac{\numg(J_0) + \ldots + \numg (J_{k-1}) + A_{k-1}}{A_0 + \cdots + A_{k -1}}
\leq \frac{A_{k-1}}{A_0 + \cdots + A_{k - 1}} + \frac{\numg(J_0)}{A_0} + \cdots + \frac{\numg(J_{k-1})}{A_{k -1}}.
\]
Since $A_i$ are non-increasing, we have
\[
\frac{A_{k - 1}}{A_0 + \cdots + A_{k-1}} \leq \frac{A_{k-1}}{kA_{k -1}} = \frac 1k < \frac \varepsilon 2.
\]
On the other hand, we have $\ord (J_i) \geq \ord(J) - i\geq N-k\geq N_0$ by the construction. So by our choice of $N_0$,
$\frac{\numg(J_i)}{A_i} < \varepsilon/(2k)$ for any $0 \leq i \leq k - 1$.
Therefore
\[
\frac {\numg (J)}{\length (A/J)}
\leq \frac{A_{k-1}}{A_0 + \cdots + A_{k - 1}} + \frac{\numg(J_0)}{A_0} + \cdots + \frac{\numg(J_{k-1})}{A_{k -1}}
< \frac \varepsilon 2 + k \frac \varepsilon {2k} = \varepsilon.
\]
This completes the proof.
\end{proof}

\begin{corollary}
\label{cor.GenvsLenRLR}
Let $(A, \mf m)$ be a regular local ring. Then for any $\varepsilon > 0$ there exists $N$ such that for any $\m$-primary ideal
$I \subseteq \mf m^N$, we have
\[
\frac {\numg (I)}{\length (A/I)} < \varepsilon.
\]
\end{corollary}
\begin{proof}
Let $\inm I = \oplus_{n \geq 0} (I \cap \m^n + \m^{n+1})/\m^{n+1}\subseteq G = \gr_{\m} (A)$ be the form ideal of $I$ (see \cite{Lech2} or \cite{HSV}). Since $A$ is regular, $G$ is a standard graded polynomial ring over a field. We have
$$\length (A/I) = \length (G/\inm I), \hspace{1em} \inm I\cdot \inm J  \subseteq \inm (IJ), \text{ and } \inm \m = G_{> 0}.$$
Thus we have
\begin{eqnarray*}
\mu (I) &=& \length (I/\mf m I) = \length (G/\inm(\m I)) - \length (G/\inm I) \\
  &\leq &\length (G/(G_{>0})\cdot\inm I) - \length (G/\inm I) = \mu_G (\inm I).
\end{eqnarray*}
Observe that if $I \subseteq \m^N$, then $\inm I \subseteq \inm \m^N = G_{> 0}^N$,
so we see that it is enough to prove the corollary for homogeneous ideals in
the polynomial ring $G$.

So now we assume $A$ is a polynomial ring over a field and $I$ is an $\m$-primary homogeneous ideal. Pick a monomial order $<$ and let $J = \inte I$. We basically repeat the above process to reduce to the monomial case. We have $\mf m^N = \inte \mf m^N \supseteq \inte I = J$, $\length (A/I)=\length ((\gr_< A)/{J})$, and that
$\mf m J = \inte \mf m  \cdot \inte I \subseteq \inte (\mf m I)$.
Therefore,
\[
\mu (I) = \length (A/\mf mI) - \length (A/I)
= \length (\frac{\gr_< A}{\inte(\mf mI)}) - \length (\frac{\gr_< A}{\inte I})\leq \length(\frac{\gr_<A}{\m J})-\length(\frac{\gr_<A}{J}) = \mu (J).
\]
Thus it is enough to treat the case of a monomial ideal $J$ in a polynomial ring $\gr_< A$, which is precisely \autoref{lem.monomial}.
\end{proof}

We next prove the main theorem on comparing the minimal number of generators and the colength for sufficiently deep submodules.

\begin{theorem}
\label{thm.GenvsLen}
Let $(A, \mf m)$ be a complete regular local ring and $L$ be a finitely generated $A$-module. Then for every $\varepsilon > 0$ there exists $N > 0$ such that for any submodule $M \subseteq \mf m^N L$ with $\length (L/M) < \infty$, we have
\[
\frac{\mu(M)}{\length (L/M)} < \varepsilon.
\]
\end{theorem}
\begin{proof}
We first prove the following claim.
\begin{claim}
\label{clm.shortexactsequence}
Let $0\to L_1\to L\to L_2\to 0$ be a short exact sequence of finitely generated $A$-modules. Suppose for every $\varepsilon > 0$ there exists $N_1, N_2 > 0$ such that for any submodule $M_i \subseteq \mf m^{N_i} L_i$ with $\length (L_i/M_i) < \infty$, we have
\[
\frac{\mu(M_i)}{\length (L_i/M_i)} < \varepsilon.
\]
Then for every $\varepsilon>0$, there exists $N$ such that for any submodule $M \subseteq \mf m^N L$ with $\length (L/M) < \infty$, we have
\[
\frac{\mu(M)}{\length (L/M)} < \varepsilon.
\]
\end{claim}
\begin{proof}[Proof of Claim]
Let $M\subseteq  L$ be a submodule. We have an induced short exact sequence:
\[ 0\to \frac{L_1}{M\cap L_1}\to \frac{L}{M}\to \frac{L}{M+L_1}\cong \frac{L_2}{\im(M)}\to 0.
\]
From the above sequence it is clear that we have:
$$\mu(L/M)\leq \mu(\frac{L_1}{M\cap L_1})+ \mu(\frac{L_2}{\im(M)}) $$
and, when $\length(L/M)<\infty$,
$$\length(L/M)=\length(\frac{L_1}{M\cap L_1})+ \length(\frac{L_2}{\im(M)}).$$
Moreover, if $N\gg0$ and $M\subseteq \m^NL$, then by the Artin-Rees lemma we know there is a constant $C$ such that $M\cap L_1\subseteq \m^{N-C}L_1$. In particular for $N\gg0$, $M\cap L_1\subseteq \m^{N_1}L_1$ and that $\im(M)\subseteq \m^{N_2}L_2$. Therefore
$$\frac{\mu(M)}{\length (L/M)}\leq \frac{\mu(\frac{L_1}{M\cap L_1})+ \mu(\frac{L_2}{\im(M)})}{\length(\frac{L_1}{M\cap L_1})+ \length(\frac{L_2}{\im(M)})}\leq \max\{\frac{\mu(\frac{L_1}{M\cap L_1})}{\length(\frac{L_1}{M\cap L_1})}, \frac{\mu(\frac{L_2}{\im(M)})}{\length(\frac{L_2}{\im(M)})}\}<\varepsilon.  \qedhere$$
\end{proof}
Now we prove the theorem. We use induction on $\dim A$. For any finitely generated $A$-module $L$, we consider a prime filtration
$$0=L_0\subseteq L_1\subseteq\cdots \subseteq L_n=L$$
where $L_{i+1}/L_i=A/Q$ for a prime ideal $Q\subseteq A$. By \autoref{clm.shortexactsequence}, it is enough to prove the theorem for each $A/Q$. Now if $Q=0$, the result follows from \autoref{cor.GenvsLenRLR}. If $Q\neq 0$, then by Cohen's structure theorem we know that $A/Q$ is a finitely generated module over a complete regular local ring $A'$ with $\dim A'<\dim A$. Note that if we view an ideal $J\subseteq A/Q$ as an $A'$-submodule, then the minimal number of generators computed over $A'$ will only possibly increase compared with the minimal number of generators computed over $A$ (while the colengths are the same). So the result follows by induction.
\end{proof}

So far our results deal with submodules that are contained in a large power of the maximal ideal times the ambient module. The next corollary improves this condition to the condition that the colength of the submodule is sufficiently large.

\begin{corollary}
\label{cor.GenvsLen}
Let $(A, \mf m)$ be a complete regular local ring and $L$ be a finitely generated $A$-module.
Then for every $\varepsilon > 0$ there exists $N > 0$ such that for any submodule $M$
such that $N < \length (L/M) < \infty$, we have
\[
\frac{\numg (M)}{\length (L/M)} < \varepsilon.
\]
\end{corollary}
\begin{proof}
We fix a $k$ that satisfies the conclusion of \autoref{thm.GenvsLen} for $\varepsilon/2$. Namely, for any
submodule $M' \subseteq \mf m^k L$ with $\length (L/M') < \infty$, we have
\[
\frac{\mu(M')}{\length (L/M')} < \varepsilon/2.
\]
Let $M' = M \cap \mf m^kL$.  We have
\[
\length (L/M) = \length (L/M') - \length (M/M') = \length (L/M') - \length ((M + \mf m^kL)/\mf m^kL)
\geq \length (L/M') - \length (L/\mf m^kL)
\]
and
\[
\numg (M) = \length (M/\mf mM) = \length (M/M') + \length (M'/\mf mM') - \length (\mf mM/\mf mM')
\leq  \numg (M') + \length (L/\mf m^kL).
\]
Because $M' \subseteq \mf m^kL$, by our choice of $k$,
\[
\frac {\numg (M')}{\length (L/M')} < \frac \varepsilon 2.
\]
Hence we have
\begin{align*}
  \frac{\numg(M)}{\length (L/M)} & \leq  \frac {\numg (M') +  \length (L/\mf m^kL)}{\length (L/M)}= \frac {\numg (M')}{\length (L/M)}+\frac {\length (L/\mf m^kL)}{\length (L/M)}\\
   & \leq  \frac {\numg (M')}{\length (L/M')}\frac{\length (L/M) + \length (L/\mf m^kL)}{\length (L/M)} +
\frac {\length (L/\mf m^kL)}{\length (L/M)}\\
   & < \frac{\varepsilon} 2 + \frac {\length (L/\mf m^kL)}{\length (L/M)} \left (1 + \frac{\varepsilon} 2\right ).
\end{align*}
Since $\length(L/\mf m^kL)$ is a constant, we choose $N\gg0$ such that $\frac {\length (L/\mf m^kL)}{N}<\frac{\varepsilon/2}{1+(\varepsilon/2)}$. The conclusion follows.
\end{proof}

The next theorem is the main result of this section.

\begin{theorem}
\label{thm.ControlSocle}
Let $(R, \mf m)$ be a Noetherian local ring. Then for any $\varepsilon > 0$ there exists $N>0$ such that for any $\mf m$-primary ideal $I$ with $\length (R/I) > N$,
\[
\frac {\length ((I:\m)/I)}{\length (R/I)} < \varepsilon.
\]
\end{theorem}
\begin{proof}
Since there is a one to one correspondence between $\m$-primary ideals in $R$ and $\widehat{R}$ and completion does not affect socle and length, we can replace $R$ by $\widehat{R}$ to assume $R$ is complete. By Cohen's structure theorem, $R$ is a quotient of a complete regular local ring $A$ (and hence a finitely generated $A$-module). Now by \autoref{cor.GenvsLen}, there exists $N_0$ such that if $N_0<\length(R/J)<\infty$, then $\frac{\mu(J)}{\length(R/J)}<\varepsilon$. We fix this $N_0$.
\begin{claim}
\label{clm.Lengthsoc}
There exists $N$ such that if $\length(R/I)>N$, then $\length(R/(I:\m))>N_0$.
\end{claim}
\begin{proof}[Proof of Claim]
We will show that $N=\length(R/\m^{N_0+1})$ works. In fact, if $\length(R/(I:\m))\leq N_0$, then we have $\m^{N_0}\subseteq I:\m$. For if $\m^{N_0}\nsubseteq I:\m$, then $\gr_{\m} (R/(I:\m))$ is nontrivial in degree $N_0$, thus the colength of $I:\m$ is at least $N_0+1$ which is a contradiction. But then we have $\m^{N_0+1}\subseteq I$ and hence $\length(R/I)\leq \length(R/\m^{N_0+1})=N$.
\end{proof}

If $I$ is an $\m$-primary ideal with $\length(R/I)>N$, \autoref{clm.Lengthsoc} shows that $\length(R/(I:\m))>N_0$, so
$$\frac {\length ((I:\m)/I)}{\length (R/I)}\leq \frac{\mu(I:\m)}{\length(R/(I:\m))}<\varepsilon$$
by our choice of $N_0$. This completes the proof.
\end{proof}

The following corollary will be a crucial ingredient in the proof of the main result in the next section.

\begin{corollary}
\label{cor.ControlColon}
Let $(R, \mf m)$ be a Noetherian local ring and $J\subset R$ a fixed $\m$-primary ideal.  Then for any $\varepsilon > 0$ there exists $N>0$ such that for any $\mf m$-primary ideal $I$ with $\length (R/I) > N$,
\[
\frac {\length ((I:J)/I)}{\length (R/I)} < \varepsilon.
\]
\end{corollary}
\begin{proof}
Since $J$ is $\m$-primary, there exists $k$ such that $\m^k\subseteq J$. Thus it is enough to prove the corollary for $J=\m^k$. We use induction on $k$. The base case $k=1$ is precisely \autoref{thm.ControlSocle}.

Now suppose the conclusion is proved for $J=\m^{k-1}$. Let $N_1$ be the number that works for $\m$ with $\varepsilon/2$, and let $N_0$ be the number that works for $\m^{k-1}$ with $\varepsilon/2$. Let $N>N_1$ be a number satisfying the conclusion of \autoref{clm.Lengthsoc}. By \autoref{clm.Lengthsoc}, for any $\m$-primary ideal $I$ such that $\length(R/I)>N$, $\length(R/(I:\m))\geq N_0$. Thus we have
\begin{align*}
  \frac{\length((I:\m^k)/I)}{\length(R/I)} & = \frac{\length((I:\m):\m^{k-1}/(I:\m))+\length((I:\m)/I)}{\length(R/I)} \\
   & \leq \frac{\length((I:\m):\m^{k-1}/(I:\m))}{\length(R/(I:\m))}+\frac{\length((I:\m)/I)}{\length(R/I)}\\
   & \leq \frac{\varepsilon}{2}+\frac{\varepsilon}{2}=\varepsilon.
\end{align*}
This finishes the proof.
\end{proof}

We end this section with an example\footnote{In fact, the anonymous referee points out that a similar and simpler example exists even in $K[[x,y]]$, with $I_{N,L}=x(x,y)^{N-1}+(y^L)$ and $J=(x)$.}  showing that the conclusion of \autoref{cor.ControlColon} fails if $J$ is not $\m$-primary.

\begin{example}
Let $R = K[[x,y,z]]$ and $I_{N, L} = \sum_{i = 1}^N (x, y)^{i}z^{N- i} + (z^L)$ where $L\gg N$. So
for each $N$, $I_{N,L}\subseteq \m^N$ (in particular, $\length(R/I_{N,L})\geq N$).
It is easy to check that $I_{N, L}:(x,y) = (x,y,z)^{N-1}$, in particular $\length (I_{N,L}: (x,y)/I_{N, L})\geq L-N$. From this we see that
\[
\lim_{L \to \infty} \frac{\length (I_{N,L}: (x,y)/I_{N,L})}{\length (R/I_{N,L})} = 1.
\]
Thus for every $N>0$, we can find $I_{N,L}$ such that $\frac{\length (I_{N,L}: (x,y)/I_{N,L})}{\length (R/I_{N,L})}$ is arbitrarily close to $1$. So the conclusion of \autoref{cor.ControlColon} cannot hold for $J=(x,y)$. Of course, the problem here is that $J$ is not $\m$-primary.
\end{example}

\section{Proof of Theorem B}

We will use the following version of the Lipman--Sathaye Jacobian theorem \cite[Theorem 2]{LipmanSathaye}. Note that full result found in \cite{LipmanSathaye} is considerably more general but with the additional assumption that $S$ is a domain. However, this condition can be replaced by $S$ being reduced and equidimensional (in particular, if $S$ is module-finite over a regular local ring $B$, then $S$ is torsion free as a $B$-module), see \cite[Theorem~3.1]{HochsterLipmanSathaye} for the generalized version.

\begin{theorem}[Lipman--Sathaye]
\label{thm.LipmanSathaye}
Let $S$ be a complete local ring that is reduced and equidimensional and let $S'$ be the normalization of $S$. Suppose $S$ is module-finite and generically \'etale over a complete regular local ring $B$. Then we have $J_{S/B}S'\subseteq S$.
\end{theorem}

We are ready to prove the main technical result of this article.

\begin{theorem}
\label{thm.mainHK}
Let $(R,\m)$ be a Noetherian local ring of characteristic $p>0$ and dimension $d\geq 1$ with $K=R/\m$ perfect. Suppose $\widehat{R}$ is reduced, equidimensional, and has an isolated singularity. Then for every $\varepsilon>0$, there exists $N\gg0$ such that for any $\m$-primary ideal $I$ with $\length(R/I)>N$, we have
$$(1-\varepsilon)\length(R/I)\leq \ehk (I)\leq (1+\varepsilon)\length(R/I).$$
\end{theorem}
\begin{proof}
We pass to the completion to assume $R$ is complete, reduced, equidimensional and has an isolated singularity. We fix a presentation $R=\frac{K[[x_1,\dots,x_n]]}{(f_1,\dots,f_t)}$ and let $J$ be the Jacobian ideal of $R$ over $K$, i.e., $J$ is the ideal generated by the $(n-d)\times(n-d)$-minors of the matrix $(\frac{\partial f_i}{\partial x_j})_{\substack{1\leq i\leq t \\ 1\leq j\leq n}}$. Since $K$ is perfect and $R$ is equidimensional, the radical of $J$ is the defining ideal of the non-singular locus of $R$ \cite[Proposition 4.5 and Lemma 4.3]{WangFittingIdeals}. Thus we know that $J$ is $\m$-primary.

We next consider $\widetilde{R}=\frac{\widetilde{K}[[x_1,\dots,x_n]]}{(f_1,\dots,f_t)}$ with $\widetilde{K}=K(a_{ij}, b_{ij})$, where $(a_{ij})_{n\times n}$ and $(b_{ij})_{t\times t}$ are matrices of new invertible indeterminates. Let $(y_1,\dots,y_n)=(x_1,\dots,x_n)(a_{ij})$ and $(g_1,\dots,g_t)=(f_1,\dots,f_t)(b_{ij})$. We claim that $\widetilde{K}[[y_{n-d+1},\dots,y_n]]\to \widetilde{R}$ is module-finite and generically \'etale (in fact, this holds for every $\widetilde{K}[[y_{i_1},\dots,y_{i_d}]]$). To see this, note that since $K$ is perfect and $R$ is equidimensional, the complete module of differentials $\widehat{\Omega}_{(R/P)/K}$ has rank $d$ for every minimal prime $P$ of $R$.\footnote{In fact, this is true even without assuming $K$ perfect, but then one must choose a coefficient field carefully, see \cite[Expos\'{e} IV, 2.1.11]{CohenGabberOriginal}.} Thus by base change
$$\rank_{\widetilde{R}}\widehat{\Omega}_{(\widetilde{R}/P)/\widetilde{K}}=\rank_R\widehat{\Omega}_{(R/P)/K}.$$
Since $\widehat{\Omega}_{(\widetilde{R}/P)/\widetilde{K}}$ is generated by $dx_1,\dots,dx_n$, any $d$ general linear combinations of them will be a basis for $\widehat{\Omega}_{(\widetilde{R}/P)/\widetilde{K}}\otimes_{\widetilde{R}/P}\kappa(P)$, where $\kappa(P)$ denotes the residue field at $P$ (which is the fraction field of $\widetilde{R}/P$). It follows that $dy_{n-d+1},\dots,dy_n$ is a basis for $\widehat{\Omega}_{(\widetilde{R}/P)/\widetilde{K}}\otimes_{\widetilde{R}/P}\kappa(P)$ for every minimal prime $P$. Therefore $\widetilde{K}[[y_{n-d+1},\dots,y_n]]\to \widetilde{R}$ is generically \'etale.

Let $A=\widetilde{K}[[y_{n-d+1},\dots,y_n]]$. Since $A\to \widetilde{R}$ is module-finite and generically \'etale, so is $A^{1/p^e}\to \widetilde{R}\otimes_AA^{1/p^e}$ for any $e$. Since $\widetilde{R}\otimes_AA^{1/p^e}$ is purely inseparable over $\widetilde{R}$, $\widetilde{R}\otimes_AA^{1/p^e}$ is equidimensional. Moreover, we have $\widetilde{R}\otimes_AA^{1/p^e}\cong A^{1/p^e}[\widetilde{R}]\subseteq \widetilde{R}^{1/p^e}$ (see \cite[Lemma 6.4]{HochsterHuneke1}) and thus $\widetilde{R}\otimes_AA^{1/p^e}$ is reduced. Now we apply \autoref{thm.LipmanSathaye} to $S=\widetilde{R}\otimes_AA^{1/p^e}$ and $B=A^{1/p^e}$, note that $\widetilde{R}^{1/p^e}$ is contained in the normalization of $\widetilde{R}\otimes_AA^{1/p^e}$ and $J_{(\widetilde{R}\otimes_AA^{1/p^e})/A^{1/p^e}}=J_{\widetilde{R}/A}\cdot (\widetilde{R}\otimes_AA^{1/p^e})$ by base change (they have the same presentation), we have $$J_{\widetilde{R}/A}\widetilde{R}^{1/p^e}\subseteq \widetilde{R}\otimes_AA^{1/p^e}.$$
Since we have a presentation $\widetilde{R}=\frac{A[y_1,\dots,y_{n-d}]}{(g_1,\dots,g_t)}$, it follows from the definition that
\begin{equation}
\label{eqn:c}
c=|\frac{\partial g_i}{\partial y_j}|_{\substack{1\leq i\leq n-d \\ 1\leq j\leq n-d}}\in J_{\widetilde{R}/A}, \text{ in particular } c\widetilde{R}^{1/p^e}\subseteq \widetilde{R}\otimes_AA^{1/p^e}.
\end{equation}

At this point we mimic the strategy of the proof of \autoref{prop.BL-2ndproof}. We consider the following two short exact sequences:
$$0\to \widetilde{R}\otimes A^{1/p^e}\to \widetilde{R}^{1/p^e}\to C_e\to 0$$
$$0\to \widetilde{R}^{1/p^e}\xrightarrow{\cdot c}\widetilde{R}\otimes A^{1/p^e}\to D_e\to 0.$$
It follows from \autoref{eqn:c} that $C_e$ and $D_e$ are annihilated by $c$ for any $e$, and $\mu(C_e)\leq \mu(\widetilde{R}^{1/p^e})$, $\mu(D_e)\leq \mu(\widetilde{R}\otimes A^{1/p^e})$. Now we tensor the two short exact sequences with $R/I$ for any $\m$-primary ideal $I$. Since $A$ is a power series ring over $\widetilde{K}$, $A^{1/p^e}$ is a finite free $A$-module of rank $p^{e\gamma}$ where $\gamma=d+\log_p[\widetilde{K}^{1/p}:\widetilde{K}]$ ($=d+n^2+t^2$), we have
$$(\widetilde{R}/I\widetilde{R})^{p^{e\gamma}}\to \widetilde{R}^{1/p^e}/I\widetilde{R}^{1/p^e}\to C_e/IC_e\to 0$$
$$\widetilde{R}^{1/p^e}/I\widetilde{R}^{1/p^e}\to(\widetilde{R}/I\widetilde{R})^{p^{e\gamma}}\to D_e/ID_e\to 0.$$
Computing length, we know that
$$p^{e\gamma}\length(\widetilde{R}/I\widetilde{R})-\length(D_e/ID_e)\leq \length(\widetilde{R}^{1/p^e}/I\widetilde{R}^{1/p^e})\leq p^{e\gamma}\length(\widetilde{R}/I\widetilde{R})+\length(C_e/IC_e),$$
where all the lengths are computed over $\widetilde{R}$. It follows that
$$p^{e\gamma}\length(\widetilde{R}/I\widetilde{R})-p^{e\gamma}\length(\widetilde{R}/(I,c)\widetilde{R})\leq \length(\widetilde{R}^{1/p^e}/I\widetilde{R}^{1/p^e}) \leq p^{e\gamma}\length(\widetilde{R}/I\widetilde{R})+\mu(\widetilde{R}^{1/p^e})\length(\widetilde{R}/(I,c)\widetilde{R}).$$
Dividing the above by $p^{e\gamma}$, we get that for any $\m$-primary $I\subseteq R$ and any $e$, we have
\begin{equation*}
\length(\widetilde{R}/I\widetilde{R})\left(1-\frac{\length(\widetilde{R}/(I,c)\widetilde{R})}{\length(\widetilde{R}/I\widetilde{R})}\right)\leq \frac{\length(\widetilde{R}^{1/p^e}/I\widetilde{R}^{1/p^e})}{p^{e\gamma}}\leq \length(\widetilde{R}/I\widetilde{R})\left(1+\frac{\mu(\widetilde{R}^{1/p^e})}{p^{e\gamma}}
\cdot\frac{\length(\widetilde{R}/(I,c)\widetilde{R})}{\length(\widetilde{R}/I\widetilde{R})}\right).
\end{equation*}
Taking limit as $e\to\infty$, we have
$$\length(\widetilde{R}/I\widetilde{R})\left(1-\frac{\length(\widetilde{R}/(I,c)\widetilde{R})}{\length(\widetilde{R}/I\widetilde{R})}\right)\leq
\ehk(I)\leq \length(\widetilde{R}/I\widetilde{R})\left(1+\ehk(R)\frac{\length(\widetilde{R}/(I,c)\widetilde{R})}{\length(\widetilde{R}/I\widetilde{R})}\right).$$
Since $I\subseteq R$, $\length_{\widetilde{R}}(\widetilde{R}/I\widetilde{R})=\length_R(R/I)$. Therefore the theorem will be proved once we established the following claim.
\begin{claim}
\label{clm.Lengthmodc}
There exists $N\gg0$ such that if $\length_R(R/I)=\length_{\widetilde{R}}(\widetilde{R}/I\widetilde{R})>N$, then $$\frac{\length(\widetilde{R}/(I,c)\widetilde{R})}{\length(\widetilde{R}/I\widetilde{R})}<\varepsilon'=\frac{\varepsilon}{\ehk(R)}.$$
\end{claim}
\begin{proof}[Proof of Claim]
Let $\overline{I}$ be the integral closure of $I$ in $R$. Since $(\overline{I},c)\widetilde{R}$ is clearly integral over $(I, c)\widetilde{R}$, by \cite[Corollary 4.2]{KMQSY} there exists a constant $D$ such that $\length(\widetilde{R}/(I,c)\widetilde{R})\leq D\length(\widetilde{R}/(\overline{I},c)\widetilde{R})$. Hence
$$\frac{\length(\widetilde{R}/(I,c)\widetilde{R})}{\length(\widetilde{R}/I\widetilde{R})}\leq D\frac{\length(\widetilde{R}/(\overline{I},c)\widetilde{R})}{\length(\widetilde{R}/\overline{I}\widetilde{R})}.$$
Thus by replacing $\varepsilon'$ by $\varepsilon'/D$, we may assume that $I$ is an $\m$-primary and integrally closed ideal of $R$. Consider the exact sequence
$$0\to\frac{I\widetilde{R}:c}{I\widetilde{R}}\to \widetilde{R}/I\widetilde{R}\xrightarrow{\cdot c} \widetilde{R}/I\widetilde{R}\to\widetilde{R}/(I,c)\widetilde{R}\to0,$$
we see that $$\length(\widetilde{R}/(I, c)\widetilde{R})=\length(\frac{I\widetilde{R}:c}{I\widetilde{R}}).$$
Now we claim that for any divisorial valuation $v$ of $R$ centered at $\m$, we have $\widetilde{v}(c)=\widetilde{v}(J\widetilde{R})$ where $\widetilde{v}$ is the extension of $v$ to $\widetilde{R}$. To see this, note that since $(g_1,\dots,g_t)$ is a general (over $R$) linear combinations of $(f_1,\dots,f_t)$, $c=|\frac{\partial g_i}{\partial y_j}|_{\substack{1\leq i\leq n-d \\1\leq j\leq n-d}}$ is a general linear combination of
$|\frac{\partial f_i}{\partial y_j}|_{\substack{i_1,\dots,i_{n-d}\\ 1\leq j\leq n-d}}$ for all $1\leq i_1<i_2<\cdots<i_{n-d}\leq t$. Therefore we have
$$\widetilde{v}(c)=\min_{i_1,\dots,i_{n-d}}\{\widetilde{v}(|\frac{\partial f_i}{\partial y_j}|_{\substack{i_1,\dots,i_{n-d}\\ 1\leq j\leq n-d}})\}.$$ But $y_1,\dots,y_n$ are general (over $R$) linear combinations of $x_1,\dots,x_n$, so each $|\frac{\partial f_i}{\partial y_j}|_{\substack{i_1,\dots,i_{n-d}\\ 1\leq j\leq n-d}}$ is a general linear combination of $|\frac{\partial f_i}{\partial x_j}|_{\substack{i_1,\dots,i_{n-d}\\ j_1,\dots,j_{n-d}}}$ for all $1\leq j_1<j_2<\cdots<j_{n-d}\leq n$. Thus
$$\widetilde{v}(|\frac{\partial f_i}{\partial y_j}|_{\substack{i_1,\dots,i_{n-d}\\ 1\leq j\leq n-d}})=\min_{j_1,\dots,j_{n-d}}\{\widetilde{v}(|\frac{\partial f_i}{\partial x_j}|_{\substack{i_1,\dots,i_{n-d}\\ j_1,\dots,j_{n-d}}})\}=\min_{j_1,\dots,j_{n-d}}\{v(|\frac{\partial f_i}{\partial x_j}|_{\substack{i_1,\dots,i_{n-d}\\ j_1,\dots,j_{n-d}}})\}.$$
Putting these together we see that
$$\widetilde{v}(c)=\min_{\substack{i_1,\dots,i_{n-d} \\ j_1,\dots,j_{n-d}}}\{ v(|\frac{\partial f_i}{\partial x_j}|_{\substack{i_1,\dots,i_{n-d}\\ j_1,\dots,j_{n-d}}})\}=v(J)=\widetilde{v}(J\widetilde{R}).$$

Finally, we note that for any $\m$-primary ideal $\mathfrak{a}\subseteq R$, $\overline{\mathfrak{a}}\widetilde{R}=\overline{\mathfrak{a}\widetilde{R}}$ by \cite[Lemma 8.4.2 (9) and Lemma 9.1.1]{HunekeSwanson}. In particular, $I\widetilde{R}$ is integrally closed and to check whether an element is in $I\widetilde{R}$ via the valuation criterion, it suffices to use divisorial valuations coming from $R$ (i.e., the Rees valuations of $I\widetilde{R}$ are extended from the Rees valuations of $I$). Since $\widetilde{v}(c)=\widetilde{v}(J\widetilde{R})$, by the valuation criterion for integral closure, we have $$I\widetilde{R}:c=I\widetilde{R}:J\widetilde{R}.$$
Since $J\widetilde{R}$ is $\m_{\widetilde{R}}$-primary, by \autoref{cor.ControlColon}, we know that there exists $N\gg0$ such that for any $I\subseteq R$ with $\length(\widetilde{R}/I\widetilde{R})>N$, we have
$$\frac{\length(\widetilde{R}/(I, c)\widetilde{R})}{\length(\widetilde{R}/I\widetilde{R})}=\frac{\length((I\widetilde{R}:c)/{I\widetilde{R}})}{\length(\widetilde{R}/I\widetilde{R})}
=\frac{\length((I\widetilde{R}:J\widetilde{R})/{I\widetilde{R}})}{\length(\widetilde{R}/I\widetilde{R})}<\varepsilon'.$$
This finishes the proof. \qedhere
\end{proof}
\end{proof}

As a consequence of \autoref{thm.mainHK}, we prove \autoref{conj.asymptoticLech} part (a) in characteristic $p>0$.

\begin{corollary}
\label{cor.mainIsolatedSing}
Let $(R,\m)$ be a Noetherian local ring of characteristic $p>0$ and dimension $d\geq 1$ with $K=R/\m$ perfect. Suppose $\widehat{R}$ has an isolated singularity. Then for every $\varepsilon>0$, there exists $N\gg0$ such that for any $\m$-primary ideal $I$ with $\length(R/I)>N$, we have $\ehk(I)\leq (1+\varepsilon)\length(R/I).$ As a consequence, $\eh(I)\leq d!(1+\varepsilon)\length(R/I).$
\end{corollary}
\begin{proof}
The second conclusion follows from the first by the inequality $\eh(I)\leq d!\ehk(I)$. To prove the inequality on Hilbert--Kunz multiplicity, we deduce it from \autoref{thm.mainHK} by removing the reduced and equidimensional hypothesis on $\widehat{R}$.

We can write a primary decomposition of $0$ in $\widehat{R}$ in the following form
\[
0 = P_1 \cap \cdots P_n \cap Q_1 \cap \cdots \cap Q_m \cap J
\]
where $P_i$ are minimal primes of dimension $d$, $Q_i$ are minimal primes of lower dimension,
and $J$ is the embedded component. Note that $J$ is necessarily $\mf m$-primary since $\widehat{R}$ has an isolated singularity. It follows that
$S = \widehat{R}/(P_1 \cap \cdots \cap P_n)$ is reduced, equidimensional, and has an isolated singularity.
Thus \autoref{thm.mainHK} can be applied to $S$. Let $N(S)$ be the number that works for $\varepsilon$ for $S$.

Now for any $\mf m$-primary ideal $I\subseteq R$ we have $\length (S/IS) \leq \length (\widehat{R}/I\widehat{R}) =\length (R/I)$ and $\ehk(IS) = \ehk(I)$ because Hilbert--Kunz multiplicity does not see lower-dimensional components.
Now if $\length (S/IS) > N(S)$ then
\[
\ehk(I) \leq (1 + \varepsilon) \length(S/IS) \leq (1 + \varepsilon) \length(R/I).
\]
Otherwise, we have that
$$\ehk(I) = \ehk(IS) \leq \ehk(S) \length(S/IS) \leq \ehk(R)N(S).$$
Hence the assertion holds for $N = \ehk(R)N(S)/(1 + \varepsilon)$.
\end{proof}

The following partial result on \autoref{conj.asymptoticLecheHK} is immediate.
\begin{corollary}
\label{cor.maineHK}
Let $(R,\m)$ be a Noetherian local ring of characteristic $p>0$ and dimension $d\geq 1$ with $K=R/\m$ perfect. Suppose $\widehat{R}$ has an isolated singularity and that $\eh(\red{\widehat{R}})>1$.
Then
\[
\lim_{N\to\infty} \sup_{\substack{\sqrt{I}=\m \\ \length(R/I)> N}} \left\{\frac{\eh(I)}{d!\length(R/I)} \right\}<\ehk(R)\leq \eh(R).
\]
\end{corollary}
\begin{proof}
\autoref{cor.mainIsolatedSing} implies that the left hand side is less than $1+\varepsilon$ for every $\varepsilon$. Therefore it is enough to show that $\ehk(R)\neq 1$. But if $\ehk(R)= 1$, then $\widehat{R}/P$ is regular for the (necessarily unique) minimal prime of $\widehat{R}$ of dimension $d$. Therefore $\eh(\red{\widehat{R}})=1$ which is a contradiction.
\end{proof}

\begin{remark}
Let $R$ be as in \autoref{cor.maineHK}, if $d\geq 2$, then we actually have $\sup_{{\sqrt{I}=\m}} \left\{\frac{\eh(I)}{d!\length(R/I)} \right\}<\ehk(R)$. This is because by \autoref{cor.maineHK}, we know there exists $\varepsilon>0$ and $N>0$ such that $\frac{\eh(I)}{d!\length(R/I)}\leq \ehk(R)-\varepsilon$ for all $I$ such that $\length(R/I)>N$. On the other hand, if $\length(R/I)\leq N$, then the set $\left\{\frac{\eh(I)}{d!\length(R/I)} \right\}_{\length(R/I)\leq N}$ is a finite set of rational numbers, each one is strictly less than $\frac{\ehk(I)}{\length(R/I)}$ by \cite[Theorem 2.2]{Hanes}, and hence strictly less than $\ehk(R)$ by \cite[Lemma 4.2]{WatanabeYoshida}. Thus by shrinking $\varepsilon$ if necessary, $\sup_{{\sqrt{I}=\m}} \left\{\frac{\eh(I)}{d!\length(R/I)} \right\}\leq \ehk(R)-\varepsilon<\ehk(R)$. One can also show that in this case $\sup_{{\sqrt{I}=\m}} \left\{\frac{\eh(I)}{d!\length(R/I)} \right\}$ is attained for some $\m$-primary ideals $I$.
\end{remark}

We next give examples to indicate the sharpness of our result. We present a counter-example to the statement of \autoref{conj.asymptoticLech} (a) without the isolated singularity assumption. We recall a simple lemma.

\begin{lemma}
Let $(R, \mf m)$ be a Noetherian local ring of dimension $d$ and let $S = R[[T]]$.
For an $\mf m$-primary ideal $I$ let $J = IS + T^{L}S$.
Then $\eh(J) = L \eh(I)$.
\end{lemma}
\begin{proof}
One may compute a power of $J$ as follows
\begin{align*}
J^n &= I^n + I^nT + I^{n}T^2 + I^{n}T^3 + \cdots + I^{n}T^{L-1}
 + I^{n-1}T^L+ I^{n-1}T^{L+1} + \cdots + I^{n-1}T^{2L - 1} + \\
&+ I^{n-2}T^{2L} + \cdots + I^{n-2}T^{3L - 1} + \cdots + IT^{(n-1)L} + \cdots IT^{nL - 1} + T^{nL} + \cdots.
\end{align*}
Thus we have
\[
\length (S/J^n) = L \sum_{k = 0}^n \length (R/I^k)
\]
and $\lim_{n \to \infty} \frac{(d+1)!\length (S/J^n)}{n^{d+1}} = L \eh(I)$.
\end{proof}

\begin{example}
\label{example}
Let $(R, \mf m)$ be a Noetherian complete local ring of dimension $d\geq 1$ such that $\eh(R) \geq 2(d+1)!$.
Consider the ring $S = R[[T]]$ and let
\[
J = J(N, L) := (\mf m^N + \mf m^{N - 1}T + \cdots + \m T^{N-1} + \mf m T^N + \cdots + \mf mT^{L-1} + T^{L})S.
\]
Observe that $J$ is an $(\m, T)$-primary ideal in $S$ and $J \subseteq (\mf m, T)^N$, so in particular $\length(S/J)\geq N$. We now estimate its colength and multiplicity as
\[
\eh(J) \geq \eh(\mf m + \mf mT + \cdots + \mf m T^{L-1} + T^{L}) = L \eh(R)
\]
and
\[
\length (S/J) = \length (R/\mf m^N) + \cdots + \length (R/\mf m^2) +L-N+1 \ll 2L \hspace{1em} \text{ for } \hspace{1em} L\gg0.
\]
Now, for any $\varepsilon > 0$ and any $N$, by taking $L \gg 0$ one will get that
\[
\eh(J)\geq L\eh(R) \geq 2L(d+1)! \geq (d+1)! (1 + \varepsilon) \length (S/J).
\]
So the upper bound in \autoref{thm.mainHK} (and \autoref{cor.mainIsolatedSing}) cannot hold in general. However, note that such $S$ cannot be an isolated singularity.
\end{example}

\begin{remark}
We observe that the lower bound in \autoref{thm.mainHK} also cannot hold in general. In \cite[Example 4.6]{Klein}, Klein gives an example of a four-dimensional normal local ring $R$ such that there exist a sequence of parameter ideals $I_n \subseteq \mf m^n$ such that
\[
\lim_{n \to \infty} \frac{\length (R/I_n)}{\eh (I_n)} \neq 1.
\]
Since Hilbert--Kunz and Hilbert--Samuel multiplicities agree for parameter ideals, $R$ cannot satisfy the conclusion of \autoref{thm.mainHK}. The $R$ constructed in \cite[Example 4.6]{Klein} is not Cohen--Macaulay on the punctured spectrum, and hence cannot be an isolated singularity.
\end{remark}

We end this section by proving the one-dimensional case of \autoref{conj.asymptoticLech} (a) in all characteristics. This will be used in the next section.

\begin{proposition}
\label{prop.dim one}
Let $(R,\m)$ be a Noetherian local ring of dimension one. Suppose $\widehat{R}$ has an isolated singularity. Then for every $\varepsilon>0$, there exists $N\gg0$ such that for any $\m$-primary ideal $I$ with $\length(R/I)>N$, we have $\eh(I)\leq (1+\varepsilon)\length(R/I).$
\end{proposition}
\begin{proof}
We can pass to the completion to assume $R$ is complete. Since $R$ is a one-dimensional isolated singularity, the nilradical of $R$ has finite length.
Thus, replacing $R$ by $\red{R}$ will not affect $\eh(I)$ and will only drop $\length(R/I)$ (by at most the length of the nilradical). Therefore we can replace $R$ by $\red{R}$ to assume that $R$ is reduced.

Let $P_1, \ldots, P_n$ be the minimal primes of $R$ and note that $R \subseteq \prod_i R/P_i$.
Let $S$ be the integral closure of $R$ in its total quotient ring $\prod_i \text{Frac}(R/P_i)$. We write $S = \prod S_i$, where $S_i$ is the integral closure
of $R/P_i$ in its field of fractions. Since $R$ is complete it follows that $S_i$ is a DVR.

We have an exact sequence $0 \to R \to S \to C \to 0$ with $\dim C = 0$, which
for each $i$ specializes to an exact sequence
\[
0 \to R/P_i \to S_i \to C_i \to 0,
\]
where $C_i$ has finite length.
Let $r_i$ be the degree of the residue field of $S_i$ over $R/\m$. By the associativity formula for multiplicity, for any $\m$-primary ideal $I \subset R$ we have
\[
\eh(I) = \sum_i \eh(I, R/P_i) = \sum_i \eh(I, S_i) = \sum_i r_i\eh(IS_i, S_i).
\]
Since $S = \prod_{i = 1}^n S_i$, it follows that
$
\length_R (S/IS) = \sum\limits_{i = 1}^n \length_R (S_i/IS_i).
$
Since $S_i$ is a DVR, we have $\length_S (S_i/IS_i) = \eh (IS_i, S_i)$
and thus
\begin{eqnarray*}
\eh(I) &=& \sum_i r_i\eh(IS, S_i) = \sum_i r_i\length_S (S_i/IS_i) =\length_R(S/IS)\\
&\le& \length_R (R/I) + \length_R (C/IC) \le \length_R(R/I) + \length_R(C).
\end{eqnarray*}
Therefore we can pick $N\gg0$ such that $\varepsilon N\geq \length(C)$. For any $\m$-primary ideal $I$ such that $\length(R/I)>N$,
we then get
$
\eh(I) \leq \length(R/I) + \length(C) \leq (1 + \varepsilon) \length (R/I)
$ as desired.
\end{proof}

\section{Proof of Theorem A}

In this section we prove the remaining part of Theorem A regarding to \autoref{conj.asymptoticLech} part (b).
We begin by recalling \cite[Lemma~2.6]{HSV}, that was stated for $x$ being a regular element, but its proof applies for a weaker assumption.

\begin{lemma}\label{lem.induction inequalities}
Let $(R, \mf m)$ be a Noetherian local ring of dimension $d > 0$, $I$ be an $\mf m$-primary ideal,
$x$ be a parameter element, and $R' = R/xR$.
If $x \notin I$, then we have
\begin{enumerate}
\item $\length (R/I) = \length (R/(I:x)) + \length (R'/I)$,
\item $\eh(I) \leq \eh(I:x) + d \eh(IR')$.
\end{enumerate}
In particular,
\[
\frac{\eh(I)}{\length(R/I)} \leq  \max \left\{\frac{\eh(I:x)}{\length (R/(I:x))}, \frac{d\eh(IR')}{\length (R'/IR')} \right\}.
\]
\end{lemma}

\autoref{lem.induction inequalities} gives us a powerful induction tool for Lech-type inequalities. It immediately provides an alternative, and simple proof of Lech's inequality. We will later generalize this strategy to prove our main result.

\begin{corollary}[\autoref{thm.Lech}]
Let $(R, \mf m)$ be a Noetherian local ring of dimension $d$.
Then for any $\mf m$-primary ideal $I$ we have $\eh(I) \leq d! \eh(R) \length (R/I)$.
\end{corollary}
\begin{proof}
We may assume that the residue field is infinite by passing to $R(t)=R[t]_{\m R[t]}$.
We use induction on $d$, where the base case $d = 0$ is trivial.

If $d = 1$, we can replace $R$ by $\widehat{R}$ and then by $\widehat{R}/\lc_{\m}^0(R)$ to assume $R$ is complete and unmixed (this doesn't change the multiplicity and will only possibly decrease the colength). Using the associativity formula we then further reduce to the case in which $R$ is a
one-dimensional complete local domain. Let $S$ be the integral closure of $R$ in its fraction field. Then $S$ is a DVR. Let $r$ be the degree of the residue field of $S$ over $R/\m$. We have
$$\eh(I)=\eh(I, S)=r\eh(IS, S)=r\length_S(S/IS)=\length_R(S/IS).$$
In particular, we have $\eh(R)=\length_R(S/\m S)$. Now by taking a filtration of $I\subseteq R$ by $R/\m$ and base change to $S$, we have $$\eh(I)=\length_R(S/IS)\leq \length_R(S/\m S)\length_R(R/I)=\eh(R)\length_R(R/I).$$

For $d\geq 2$ we also use induction on $\length (R/I)$.
The inequality clearly holds for $I = \mf m$.
For arbitrary $I$ we use \autoref{lem.induction inequalities}
for a superficial element $x\in\mf m$ (thus $\eh(R) = \eh(R')$ by \cite[Proposition 8.5.7 and 11.1.9]{HunekeSwanson}).
Then $\eh(I:x)/\length (R/(I:x)) \leq d! \eh(R)$ by the induction on the colength
and $d\eh(IR')/\length (R'/IR') \leq d! \eh(R)$ by the induction on the dimension.
\end{proof}

Now we start the proof of \autoref{conj.asymptoticLech} (b) in equal characteristic. We first show that the ``if" direction hold, in arbitrary characteristic.
\begin{proposition}
\label{prop.if direction}
Let $(R,\m)$ be a Noetherian local ring of dimension $d\geq 1$. Suppose
\[
\lim_{N\to\infty} \sup_{\substack{\sqrt{I}=\m \\ \length(R/I)> N}} \left\{\frac{\eh(I)}{d!\length(R/I)} \right\}<\eh(R).
\]
Then we have $\eh(\red{\widehat{R}}) > 1$.
\end{proposition}
\begin{proof}
Without loss of generality, we may assume that $R$ is complete. If $\eh(\red{R}) = 1$, then consider the family of ideals $I_n = \mf m^n + \sqrt{0}$.
Since $\m^n\subseteq I_n \subseteq \overline{\mf m^n}$, we have $\eh(I_n) = \eh(\m^n)=n^d \eh(R)$. On the other hand, $\length(R/I_n)=\length(\red{R}/\m^n\red{R})$, so as $n$ tends to infinity it tends to $n^d/d!$ since $\eh(\red{R})=1$. Therefore
\[
\lim_{n \to \infty} \frac{\eh(I_n)}{d! \length (R/I_n)}
=\lim_{n \to \infty} \frac{n^d\eh(R)}{d! \length (R/I_n)} = \eh(R)
\]
which is a contradiction.
\end{proof}

The next lemma originates from \cite[Korollar~4.2]{Flenner} where the assumption on dimension was missing.
We present a proof here for completeness.

\begin{lemma}\label{lem.Flenner}
Let $(R, \mf m)$ be an equal characteristic Noetherian complete local ring of dimension $d$
such that $R/\mf m$ is an infinite field.
Let $s\leq d-2$ and suppose $R$ satisfies $(R_s)$, i.e.,
$R_P$ is regular for all primes of height at most $s$.
Then for a general element $x$ of $\mf m$, $R/xR$ also satisfies $(R_s)$.
\end{lemma}
\begin{proof}
Flenner \cite[Satz~4.1]{Flenner} shows that a general element $x \in \mf m$ satisfies that
$x \notin P^{(2)}$ for any $P \neq \mf m$.
We will also demand that $x \notin Q$ for any minimal prime ideal $Q$ of the defining ideal of the singular locus $J$ of height $s+1$ (note that
$\hght J \geq s+1$ by our assumption, and if $J$ has height $\geq s+2$, then this condition is empty). This is possible because $s+1<d$ by our assumption.

Note that if the image of a prime ideal $P$ containing $x$ is a height $h$ prime of $R'$, then $P$ has height $h + 1$. Since $x \notin P^2R_P$, if $R_P$ is regular then $R_P/xR_P$ is also regular. This is automatic if $h < s$. If $h = s$, then since $x \in P$, it follows that $J$ is not contained in $P$ (since otherwise $\hght P \geq s + 2$) and hence $R_P/xR_P$ is still regular.
\end{proof}

The following lemma is also well-known. But we include a short proof for completeness.
\begin{lemma}\label{lem.num bound}
Let $(R, \mf m)$ be a Noetherian local ring of dimension one.
Then the number of generators of any $\mf m$-primary ideal
is bounded by $\eh(R) + \length (\lc^0_{\mf m} (R))$.
\end{lemma}
\begin{proof}
Observe that $S := R/\lc^0_{\mf m} (R)$ is a Cohen--Macaulay ring. Then
\[
\mu(I) = \length (I/\mf mI)\leq \mu(IS) + \length ((\lc^0_{\mf m} (R) + \mf mI)/\mf m I)
\leq  \mu(IS)+ \length (\lc^0_{\mf m} (R)).
\]
We may extend the residue field of $R$ and $S$ without changing $\mu(IS)$ or $\eh(S) = \eh(R)$.
Thus we may assume $\mf m$ has a minimal reduction $x$. Now for any $\mf m$-primary ideal $I$, we have
\[
\mu (IS) = \length (IS/\mf mIS) \leq \length (IS/xIS) = \eh(x, IS)= \eh(x, S) = \eh(R). \qedhere
\]
\end{proof}

We next recall the following result of Hanes (\cite[Theorem~2.4]{Hanes}).

\begin{theorem}[Hanes]
\label{thm.Hanes}
Let $(R, \mf m)$ be a Noetherian local ring of characteristic $p > 0$ and dimension $d > 1$.
Then for any $\mf m$-primary ideal $I$ we have
\[
\eh(I) \leq d! \frac{(\mu(I)^{1/(d-1)} -1)^{d-1}}{\mu(I)}\ehk(I) = d! \left (1 - \frac{1}{\mu(I)^{1/(d-1)}}\right)^{d-1} \ehk(I).
\]
\end{theorem}

\begin{proposition}
\label{prop.LechbddGen}
Let $(R,\m)$ be an equal characteristic Noetherian complete local ring of dimension $d>1$. Then for any $\m$-primary ideal $I$ such that $\mu(I)\leq C$, we have
$$\eh(I)\leq d!\left(1-\frac{1}{C^{1/(d-1)}}\right)^{d-1}\eh(R)\length(R/I).$$
\end{proposition}
\begin{proof}
If $R$ has characteristic $p>0$, then the assertion follows from \autoref{thm.Hanes} and the fact that $\ehk(I)\leq \ehk(R)\length(R/I)\leq \eh(R)\length(R/I)$ (for example see \cite[Lemma 4.2]{WatanabeYoshida}).

If $R$ has characteristic $0$, we prove it using Artin approximation and reduction mod $p>0$. We give the idea and omit the technical details here. Suppose we have a counter-example in characteristic $0$, then we think of the counter-example as a pair $(R,I)$, where $R$ is a finitely generated module over a complete regular local ring $A$ that has an algebra structure, $I\subseteq R$ is an $A$-submodule that is also an $R$-submodule whose number of generators over $R$ is $\leq C$. All these data can be described by using equations over $A$. We next note that one can keep track of $\eh(I)$, $\eh(R)$ and $\length(R/I)$ using equations, for example see \cite[5.1]{MaLech}.\footnote{\cite[5.1]{MaLech} only explains this when $I=\m$, but the same argument works for arbitrary $\m$-primary ideal $I$. } Therefore by Artin approximation, the counter-example descends to a counter-example over a henselian regular local ring $A'$, and thus descends to a counter-example essentially of finite type over $K=R/\m$. We then use standard reduction mod $p>0$ technique to obtain a counter-example in characteristic $p>0$, by noting that $\eh(I)$ and $\eh(R)$ can be computed by alternating sum of the lengths of Koszul homology modules of a minimal reduction of $I$ and $\m$ respectively, and by picking a suitable model and the generic flatness, these lengths at the generic fiber are the same as the special fiber for $p\gg0$ (see \cite[5.2]{MaLech}). Therefore we eventually arrive a counter-example in characteristic $p>0$, which is a contradiction.
\end{proof}

We now prove the main result of this section. Compared with \autoref{thm.Lech}, this result says that, under very mild assumptions on $R$,
in dimension at least two Lech's inequality can be improved uniformly for all $\m$-primary ideals $I$. The rough strategy is to use \autoref{lem.induction inequalities} and \autoref{lem.Flenner} to reduce to the case that $\dim R=2$, and then combine the previous results to handle the two-dimensional case.

\begin{theorem}
\label{thm.Mainequalchar}
Let $(R, \mf m)$ be an equal characteristic Noetherian complete local ring of dimension $d \geq 2$. Suppose $R$ satisfies $(R_0)$ and that $\eh(R)>1$.
Then there exists $\varepsilon > 0$ such that for any $\m$-primary ideal $I$, we have
\[
\eh(I) \leq d!(\eh(R) - \varepsilon) \length (R/I).
\]
\end{theorem}
\begin{proof}
We may pass to $R(t)=R[t]_{\m R[t]}$ to assume that $R$ has an infinite residue field. Let $x \in \mf m$ be a general element and $R' = R/xR$. We note that $R'$ still satisfies $(R_0)$ by \autoref{lem.Flenner}, and since $d\geq 2$, we have $\eh(R) = \eh(R')$ by \cite[Proposition 8.5.7 and 11.1.9]{HunekeSwanson}.

We use induction on $d$ and we first show the inductive step. So, we assume $d\geq 3$ and the result holds for $R'$. That is, there exists $\varepsilon$ such that $\eh(J) \leq (d-1)!(\eh(R') - \varepsilon) \length (R'/J)$ for any $\m$-primary ideal $J$ in $R'$. We use induction on $\length (R/I)$ to show that the same $\varepsilon$ works for $R$ (the initial case $I=\m$ is obvious). By \autoref{lem.induction inequalities} we have
\[
\frac{\eh(I)}{d!\length (R/I)}
\leq \max \left\{\frac{\eh(I:x)}{d!\length (R/(I:x))}, \frac{\eh(IR')}{(d-1)!\length (R'/IR')} \right\}
\leq \eh(R') - \varepsilon=\eh(R)-\varepsilon.  \]

It remains to prove the base case $d=2$. Let $x \in \mf m$ be a general element and $R' = R/xR$. We note that $R'$ still satisfies $(R_0)$ by \autoref{lem.Flenner}.
Fix any $\varepsilon_0 > 0$ such that $\eh(R) - \varepsilon_0 > 1$.
By \autoref{prop.dim one} we can find $N$ such that
$\eh(J) \leq (\eh(R) - \varepsilon_0) \length (R'/J)$ for every ideal $J \subseteq R'$ such that
$\length (R'/J)> N$.

Now, suppose that $\length (R'/IR') \leq N$ (e.g., $\length (R/I) \leq N$).
By \cite[Lemma~2.2]{Goto} or \cite[Theorem~1]{WatanabeJ}, $\mu(I) \leq \mu(IR') + \length(R'/IR')$
and by \autoref{lem.num bound} this implies that $\mu(I)$ is bounded by a constant $C$ that only depends on $R$ and $R'$.
Therefore by \autoref{prop.LechbddGen},
\[
\eh(I) \leq 2 \left (1 - \frac{1}{C} \right) \eh(R)\length (R/I).
\]
Thus we can find $\varepsilon_{1}$ such that $\eh(I) \leq 2 (\eh(R) - \varepsilon_1) \length(R/I)$.

Finally we use induction on $\length (R/I)$ to show that $\varepsilon = \min (\varepsilon_0, \varepsilon_1)$ works for  all $I$.
We may assume that $\length (R'/IR') > N$. Then by \autoref{lem.induction inequalities} we have
\[
\frac{\eh(I)}{2\length (R/I)}
\leq \max \left\{\frac{\eh(I:x)}{2\length (R/(I:x))}, \frac{\eh(IR')}{\length (R'/IR')} \right\}
\leq \eh(R) - \varepsilon.  \qedhere
\]
\end{proof}

In dimension $\geq 2$, \autoref{conj.asymptoticLech} (b) follows from \autoref{thm.Mainequalchar} (in equal characteristic):
\begin{corollary}[Uniform Lech's inequality]
\label{cor.Mainequalchar}
Let $(R, \mf m)$ be an equal characteristic Noetherian local ring of dimension $d \geq 2$. Suppose $\eh(\red{\widehat{R}})>1$. Then there exists $\varepsilon > 0$ such that for any $\m$-primary ideal $I$, we have
$
\eh(I) \leq d!(\eh(R) - \varepsilon) \length (R/I).
$
In particular,
$$\lim_{N\to\infty} \sup_{\substack{\sqrt{I}=\m \\ \length(R/I)> N}} \left\{\frac{\eh(I)}{d!\length(R/I)} \right\}\leq \eh(R)-\varepsilon<\eh(R).$$
\end{corollary}
\begin{proof}
Since completion does not affect colength and multiplicity, we may assume that $R$ is a complete local ring of dimension $d\geq 2$.
Let $P_1,\ldots,P_n$ be minimal primes of $R$ such that $\dim(R/P_i)=d$.
By the associativity formula for multiplicity, we have $$\eh (I\red{R}) = \sum_{i = 1}^n \eh(I, R/P_i) \hspace{1em} \text{ and } \hspace{1em}  \eh (\red{R}) = \sum_{i = 1}^n \eh(R/P_i).$$ Thus by applying \autoref{thm.Mainequalchar} to $\red{R}$, we know there exists $\varepsilon$ such that for any $\m$-primary ideal $I$,
\[
\sum_{i = 1}^n \eh(I, R/P_i)  \leq d!(\sum_{i = 1}^n \eh(R/P_i)  - \varepsilon)\length(\red{R}/I\red{R})
\leq d!(\sum_{i = 1}^n \eh(R/P_i)  - \varepsilon)\length(R/I).
\]
Therefore for any $\m$-primary ideal $I$ there exists $k$ such that
\[
\eh(I, R/P_k) \leq d!(\eh(R/P_k)-\varepsilon/n)\length(R/I).
\]
For $i\neq k$, by Lech's inequality we know that
$
\eh(I, R/P_i) \leq d!\eh(R/P_i)\length(R/I).
$
Thus by the associativity formula for multiplicity
\begin{align*}
\eh(I) &= \sum_{i = 1}^n \eh(I, R/P_i) \length (R_{P_i})\\
&\leq \sum_{i = 1}^n d! \eh(R/P_i)\length(R/I)\length (R_{P_i}) -d!\frac{\varepsilon}{n} \length (R_{P_k})\length(R/I)\\
&=d! \left (\eh(R) - \frac{\varepsilon}{n}\length (R_{P_k}) \right) \length (R/I).
\end{align*}
Therefore by setting $\varepsilon' = \frac{\varepsilon}{n} \min_i \length (R_{P_i})>0$, we see that $\eh(I)\leq d! \left (\eh(R) - \varepsilon') \right) \length (R/I)$ for any $\m$-primary ideal $I$.
\end{proof}

It remains to prove \autoref{conj.asymptoticLech} (b) in dimension one. We point out the following fact which is of independent interest.

\begin{lemma}
\label{lem.lengthOrderequiv}
Let $(R,\m)$ be a Noetherian local ring of dimension $d\geq 1$. Then we have:
$$\lim_{N\to\infty} \sup_{\substack{\sqrt{I}=\m \\ \length(R/I)> N}} \left\{\frac{\eh(I)}{d!\length(R/I)} \right\}=\lim_{N\to\infty} \sup_{\substack{\sqrt{I}=\m \\ I\subseteq \m^N}} \left\{\frac{\eh(I)}{d!\length(R/I)} \right\}.$$
\end{lemma}
\begin{proof}
If $I\subseteq \m^N$, then clearly $\length(R/I)\geq N$ so ``$\geq$" is obvious. Now we fix $N$ and we let $\lambda=\sup_{\substack{\sqrt{I}=\m \\ I\subseteq \m^N}} \left\{\frac{\eh(I)}{d!\length(R/I)} \right\}$. For any $I$ let $I'=I\cap \m^N$. Then we have $\eh(I)\leq \eh(I')$ while
$$\length(R/I)=\length(R/I')-\length(I/I')=\length(R/I')-\length(\frac{I+\m^N}{\m^N})\geq \length(R/I')-\length(R/\m^N).$$
Since $I'\subseteq \m^N$, we have
$$\frac{\eh(I)}{d!\length(R/I)}\leq \frac{\eh(I')}{d!(\length(R/I')-\length(R/\m^N))}\leq \lambda\cdot \left(\frac{\length(R/I')}{\length(R/I')-\length(R/\m^N)}\right).$$
Note that when $\length(R/I)\to\infty$, $\length(R/I')\to\infty$ and hence $\frac{\length(R/I')}{\length(R/I')-\length(R/\m^N)}\to 1$. Thus
$$\lim_{N\to\infty} \sup_{\substack{\sqrt{I}=\m \\ \length(R/I)> N}} \left\{\frac{\eh(I)}{d!\length(R/I)} \right\}\leq \lambda. \qedhere $$
\end{proof}

Finally we prove \autoref{conj.asymptoticLech} (b) in dimension one, in all characteristics. In fact, we can completely understand this asymptotic invariant in dimension one.

\begin{proposition}
\label{prop.MaindimOne}
Let $(R,\m)$ be a Noetherian local ring of dimension one. Then we have
$$\lim_{N\to\infty} \sup_{\substack{\sqrt{I}=\m \\ \length(R/I)> N}} \left\{\frac{\eh(I)}{d!\length(R/I)} \right\}=l:=\max\{\length(\widehat{R}_{P_i})| P_i \text{ is a minimal prime of $\widehat{R}$}\}.$$
In particular, if $\eh(\red{\widehat{R}})>1$, then
$$\lim_{N\to\infty} \sup_{\substack{\sqrt{I}=\m \\ \length(R/I)> N}} \left\{\frac{\eh(I)}{\length(R/I)} \right\}<\eh(R).$$
 \end{proposition}
\begin{proof}
Since completion does not affect colength and multiplicity, we may assume $R$ is complete. Note that by \autoref{prop.dim one}, for any $\varepsilon>0$, there exists $N\gg0$ such that if $I\subseteq \m^N$ (and hence $I\red{R}\subseteq \m^N\red{R}$), then
$$\eh(I\red{R})\leq (1+\varepsilon)\length(\red{R}/I\red{R})\leq (1+\varepsilon)\length(R/I).$$
Let $P_1,\dots, P_n$ be the minimal primes of $R$. By the associativity formula for multiplicity,
$$\eh(I)=\sum_{i=1}^{n}\length(R_{P_i})\eh(I, R/P_i)\leq l\cdot \sum_{i=1}^{n}\eh(I, R/P_i)=l\cdot \eh(I\red{R})\leq l\cdot (1+\varepsilon)\length(R/I).$$
Therefore by \autoref{lem.lengthOrderequiv}, we have
$$\lim_{N\to\infty} \sup_{\substack{\sqrt{I}=\m \\ \length(R/I)> N}} \left\{\frac{\eh(I)}{\length(R/I)} \right\}=\lim_{N\to\infty} \sup_{\substack{\sqrt{I}=\m \\ I\subseteq \m^N}} \left\{\frac{\eh(I)}{\length(R/I)} \right\}\leq l.$$
On the other hand, we know $l=\length(R_{P_i})$ for some $i$. Consider the ideal $I_N=P_i+\m^N$. Then clearly $\length(R/I_N)\geq N$ and we have
$$\eh(I_N)=\sum_{i=1}^{n}\length(R_{P_i})\eh(I_N, R/P_i)\geq l\cdot \eh(\m^N, R/P_i)=l\cdot N\eh(R/P_i).$$
Therefore we have
$$\lim_{N\to\infty} \sup_{\substack{\sqrt{I}=\m \\ \length(R/I)> N}} \left\{\frac{\eh(I)}{\length(R/I)} \right\}\geq \lim_{N\to\infty}\frac{\eh(I_N)}{\length(R/I_N)}\geq \lim_{N\to\infty}\frac{l\cdot N\eh(R/P_i)}{\length(R/(P_i+\m^N))}=l.$$
Finally, if $\eh(\red{R})>1$, then either $n\geq 2$ or $\eh(R/P_i)>1$, so in either case we have $l<\eh(R)$. Thus the last assertion follows.
\end{proof}

\begin{remark}
\begin{enumerate}
\item In the case $d\geq 2$, the conclusion of \autoref{cor.Mainequalchar} is stronger than what \autoref{conj.asymptoticLech} (b) predicts as it shows that $\frac{\eh(I)}{d!\length(R/I)}$ is uniformly bounded away from $\eh(R)$ for {\it any} $\m$-primary ideal, while \autoref{conj.asymptoticLech} (b) only expects this for sufficiently deep ideals. However, we point out that these are actually equivalent statements. Suppose one knows $\frac{\eh(I)}{d!\length(R/I)}\leq \eh(R)-\varepsilon$ for any $\m$-primary ideal $I$ with $\length(R/I)>N$. Then since $\{\frac{\eh(I)}{d! \length (R/I)}\}_{\length(R/I)\leq N}$ is a finite set of rational numbers with a bounded denominator and each is strictly less that $\eh(R)$ by the non-sharpness of Lech's inequality in dimension $\geq 2$ (see \cite[page 74, after (4.1)]{LechMultiplicity}),  we know there exists $\varepsilon'$ such that $\frac{\eh(I)}{d!\length(R/I)}\leq \eh(R)-\varepsilon'$ for any $\m$-primary ideal $I$.
\item If $R$ has characteristic $p>0$ and $\dim R=2$, the method we used in the proof of \autoref{thm.Mainequalchar} and \autoref{cor.Mainequalchar} can be adapted to prove \autoref{conj.asymptoticLecheHK} if $\dim R=2$. We omit the details and leave this to the interested reader. Note that, however, one cannot expect to use the same strategy to prove \autoref{conj.asymptoticLecheHK} in the higher dimensional cases. Because it is not true in general that we can find an element $x\in R$ such that $\ehk(R)=\ehk(R/xR)$. For example if we let $R_d = K[[x_1, \ldots, x_d]]/(x_1^2 + \cdots +x_d^2)$, then $R/xR$ is isomorphic $R_{d-1}$ for a linear form $x$, but $\ehk (R_{3}) > \ehk (R_4)$, see \cite[\S 4]{WatanabeYoshidaHKdimThree}.
\end{enumerate}
\end{remark}

As explained in \cite[Remark 3.4]{DaoSmirnov}, an improvement in Lech's inequality gives an improved Lech-type inequality
on the number of generators of an integrally closed ideal.

\begin{corollary}
\label{cor.numgen}
Let $(R, \mf m)$ be a Noetherian local ring of dimension $d$ with infinite residue field, $x$ denote a general element in $\mf m$,
and $I$ be an $\mf m$-primary integrally closed ideal.
\begin{enumerate}
\item If $R$ is equicharacteristic, $d \geq 3$, and $\eh(\red{\widehat{R}})>1$,
then there exists $\varepsilon > 0$ independent of $I$ such that
\[
\eh(IR/xR) \leq (d-1)!(\eh(R) - \varepsilon) (\mu(I) -d + 1).
\]
\item If $\widehat{R}$ has an isolated singularity of characteristic $p > 0$, $d \geq 2$, and $R/\mf m$ is a perfect field,
then for every $\varepsilon>0$, there exists $N\gg0$ such that if $\mu(R/I)>N$ then
\[
\eh(IR/xR) \leq (d-1)!(1+\varepsilon)(\mu(I) -d + 1).
\]
\end{enumerate}
\end{corollary}
\begin{proof}
The proof of both results follows from the same argument as in \cite[Theorem 3.1]{DaoSmirnov}, which is based on Lech's inequality in $R/xR$ and
a formula of Watanabe: $\mu(I) = \mu (IR/xR) + \length (R/(I,x))$ (see \cite{WatanabeJ}).
For the second assertion we observe that $R/xR$ is still an isolated singularity by the proof of Lemma~\ref{lem.Flenner}
and that $\mu(I) \leq 2 \length (R/(I,x))$ by Watanabe's formula.
\end{proof}

\bibliographystyle{plain}
\bibliography{refs}

\end{document}